\documentclass[a4paper,11pt,oneside]{article}
\pdfoutput=1
\RequirePackage{amsmath, amsthm, amssymb, mathtools, mathrsfs}
\RequirePackage{natbib}
\RequirePackage{tikz,tikz-cd}
\usetikzlibrary{patterns}

\usepackage[margin=1in]{geometry}
\usepackage{libertine}
\usepackage{courier}
\usepackage{mathpazo}
\usepackage{xcolor}
\usepackage{pifont}
\usepackage{authblk}
\usepackage[backref=page]{hyperref}
\usepackage{enumerate}

\usepackage{titlesec}

\definecolor{wred}{rgb}{0.533333,0.10980,0.15686}
\definecolor{wredlight}{rgb}{0.788, 0.11, 0.157}

\definecolor{darkbrown}{RGB}{170,100,0} 
\definecolor{darkdarkbrown}{RGB}{110,70,0} 
\titleformat{\section}
{\color{wred}\normalfont\Large\bfseries}
{\color{wred}\thesection}{1em}{}
\titleformat{\subsection}
{\color{wred}\normalfont\Large\bfseries}
{\color{wred}\thesubsection}{1em}{}

\hypersetup{
	colorlinks=true,
  linkcolor=wredlight,
  citecolor=wredlight,
  filecolor=wredlight,
  urlcolor=wredlight
}

 \renewcommand*{\backrefalt}[4]{%
    \ifcase #1%
     \or (page:~#2)%
     \else (pages:~#2)%
    \fi%
    }

    \makeatletter
\def\@fnsymbol#1{\ensuremath{\ifcase#1\or \dagger\or \ddagger\or
   \mathsection\or \mathparagraph\or \|\or **\or \dagger\dagger
   \or \ddagger\ddagger \else\@ctrerr\fi}}
    \makeatother

\newtheorem{theorem}{Theorem}[section]
\newtheorem{lemma}{Lemma}[section]
\newtheorem{definition}{Definition}[section]
\newtheorem{proposition}{Proposition}[section]

\newtheorem{example}{Example}[section]
\newtheorem{corollary}{Corollary}[section]
\newtheorem{remark}{Remark}[section]

\providecommand{\keywords}[1]
{
  \small	
  \textbf{\textit{Keywords---}} #1
}

\DeclareMathOperator*{\argmin}{arg\,min}

\DeclareMathOperator{\polylog}{polylog}
\DeclareMathOperator{\Cov}{Cov}
\newcommand{\pred}[1]{\delta\left[#1\right]}
\newcommand{\set}[1]{\left\{ #1 \right\}}
\DeclarePairedDelimiter\ceil{\lceil}{\rceil}
\DeclarePairedDelimiter\floor{\lfloor}{\rfloor}

\newcommand{\eqdef}{\doteq}
\newcommand{\bbR}{\mathbb{R}}
\newcommand{\bbP}{\mathbb{P}}

\newcommand{\bbE}{\mathbb{E}}
\newcommand{\bbN}{\mathbb{N}}
\newcommand{\bbZ}{\mathbb{Z}}
\newcommand{\trn}{^\intercal}
\newcommand{\abs}[1]{\left| #1 \right|}
\newcommand{\nrm}[1]{\left\Vert #1 \right\Vert}
\newcommand{\inrm}[1]{\| #1 \|}
\newcommand{\PR}[2][]{\mathbb{P}_{#1}\left( #2 \right)}
\newcommand{\E}[2][]{\mathbb{E}_{#1}\left[ #2 \right]}
\newcommand{\Var}[2][]{\operatorname{Var}_{#1}\left[ #2 \right]}
\newcommand{\eps}{\varepsilon}
\newcommand{\bigO}{\mathcal{O}}

\newcommand{\calX}{\mathcal{X}}

\newcommand{\calP}{\mathcal{P}}
\newcommand{\calE}{\mathcal{E}}

\newcommand{\calB}{\mathcal{B}}

\newcommand{\calW}{\mathcal{W}}

\newcommand{\calG}{\mathcal{G}}
\newcommand{\calJ}{\mathcal{J}}

\newcommand{\calS}{\mathcal{S}}
\newcommand{\calA}{\mathcal{A}}

\newcommand{\tv}[1]{\nrm{#1}_{\mathsf{TV}}}
\DeclareMathOperator{\Spec}{Spec}
\newcommand{\tmix}{t_{\mathsf{mix}}}
\newcommand{\trel}{t_{\mathsf{rel}}}
\newcommand{\atmix}{t^{\sharp}_{\mathsf{mix}}}
\newcommand{\abtmix}{\overline{t}^{\sharp}_{\mathsf{mix}}}
\newcommand{\Cue}{C_{\mathsf{u.erg.}}}
\newcommand{\Ce}{C_{\mathsf{erg.}}}

\newcommand{\appenref}[1]{#1}

\newtheorem{condition}{Condition}

\title{\vspace{-2.0cm} Optimistic Estimation of Convergence in \\ Markov Chains with the Average-Mixing Time}

\begin{document}

\author[1]{Geoffrey Wolfer \thanks{email: wolfer@go.tuat.ac.jp. 
}}
\affil[1]{Department of Electrical Engineering and Computer Science, Tokyo University of Agriculture and Technology, Tokyo}
\author[2]{Pierre Alquier \thanks{email: alquier@essec.edu.}}
\affil[2]{ESSEC Business School, Asia-Pacific campus, Singapore}
\date{\today}

\maketitle

\begin{abstract}
    The convergence rate of a Markov chain to its stationary distribution is typically assessed using the concept of total variation mixing time. However, this worst-case measure often yields pessimistic estimates and is challenging to infer from observations. In this paper, we advocate for the use of the average-mixing time as a more optimistic and demonstrably easier-to-estimate alternative. We further illustrate its applicability across a range of settings, from two-point to countable spaces, and discuss some practical implications.
\end{abstract}


\tableofcontents

\section{Introduction}
\label{section:introduction}
The total variation distance variant of the mixing time $\tmix$ of a Markov chain $X_1, \dots, X_n$ on a discrete state space $\calX$, governed by a transition operator $P$, quantifies the time to convergence to its stationary distribution $\pi$. 
Two observations on the chain separated in time by about $\tmix$ will be approximately independent and sampled from $\pi$.
Estimates on the mixing time yield exponential concentration inequalities \citep{chung2012chernoff, paulin2015concentration}, and in turn, generalization bounds for stable machine learning algorithms trained with Markovian data \citep{mohri10a,agarwal2012generalization, oliveira2022split}, regret bounds for reinforcement learning algorithms \citep{ortner2020regret}, and convergence diagnostics for Markov Chain Monte Carlo methods, which are widely used in mathematical physics \citep{metropolis1953equation, hastings1970monte}, computer science \citep{sinclair2012algorithms} and statistics \citep{chib1995understanding}.

However, $\tmix$ is an intrinsically pessimistic notion, defined in terms of the worst-case initial distribution.
Furthermore,
$\tmix$ is generally unknown a priori, and
theoretical upper bounds, when available, are typically conservative. Finally, estimation of $\tmix$ from a single trajectory of observations of length $n$ is known to be a statistically hard problem which requires at least
\begin{equation}
\label{eq:lower-bound-tmix-estimation}
    n \in \Omega \left( \frac{1}{\pi_\star} + \tmix\abs{\calX} \right),
\end{equation}
for estimation down to constant multiplicative error \citep[Theorem~3.1, Theorem~3.2]{hsu2019mixing}, where $\abs{\calX}$ is the size of the space and $\pi_\star$ is the minimum stationary probability (refer to Section~\ref{section:notation}).
The sample complexity lower bound in \eqref{eq:lower-bound-tmix-estimation} is prohibitively large when $\pi_\star \propto \exp(-\abs{\calX})$ for vast state spaces, and even vacuous in the infinite case.

In this paper, we propose to address some of the above listed concerns about the worst-case $\tmix$, by considering the average-mixing time $\tmix^\sharp$ which recently came under the spotlight in \citet{espuny2022speeding, munch2023mixing}. 
Specifically, we will first confirm that $\tmix^\sharp$ is a more optimistic measure of convergence, in the sense where some chains can average-mix arbitrarily faster than they mix, even on a two-point space.
We then initiate the program of estimating $\tmix^\sharp$ from a single trajectory of observations and we demonstrate that the problem is statistically less costly than the one pertaining to $\tmix$. In fact, contrary to the worst-case problem, our analysis carries to a large subset of countable-state Markov chains.
One should rightfully expect $\tmix^\sharp$ to be less informative than $\tmix$ since it measures a weaker notion of mixing. However, we confirm that knowledge of an envelope average-mixing time (pertaining to the $\beta$-mixing coefficients of the process) remains of practical value in the sense where it can be used in known powerful decoupling techniques \citep{yu1994rates} and appears as a natural parameter controlling the large deviation rate for additive functionals.
Along the way, we explore the connections between the average-mixing time and $\beta$-mixing, argued by \citet{vidyasagar2013learning, mohri10a} to be
 ``just the right'' assumption  for the
analysis of machine learning algorithms with weakly-dependent data.

\subsection{Related work}
\label{section:related-work}
Concentration inequalities for Markov chains typically involve a parameter, such as a spectral gap or the mixing time, which accounts for dependencies within the sample \citep{lezaud2001chernoff, a2004optimal, 10.1214/ECP.v20-3966, 10.1214/20-ECP286}. 
Estimates on this parameter yield corresponding deviation bounds. However, this parameter is often unknown and its definition is restricted to geometric ergodic settings.

\subsubsection{Estimation of mixing time in Markov chains}
The program of estimating the (worst-case) mixing time of a Markov chain from a single trajectory of observations was initiated by \citet{hsu2015mixing}. In the time-reversible setting, they obtained the first upper and lower bounds on the sample complexity of estimating the absolute spectral gap of a chain in multiplicative error, and constructed fully empirical confidence intervals for the quantity.
The analysis in the reversible setting was later complemented in \citet{hsu2019mixing}, introducing an amplification method to obtain a generally unimprovable upper bound on the problem.
In the more challenging non-reversible setting, the absolute spectral gap is no longer a good proxy for the mixing time. In \citet{wolfer2024improved}, the authors propose to estimate Paulin's pseudo-spectral gap \citep{paulin2015concentration} instead, and devise a method which enjoys the same sample complexity as in the reversible setting, modulo logarithmic factors. In parallel, \citet{wolfer2024empirical} put forward a contraction approach, using a generalization of Dobrushin's coefficient, instead of spectral methods to construct fully empirical confidence intervals for the mixing time.

\subsubsection{Average-mixing time}
Average variants of the mixing time have recently made their apparition in the literature.
For instance, \citet{berestycki2018random} analyzed random walks on the giant component of the 
Erd\"{o}s--R\'{e}nyi
random graph $\calG(n, p)$ with $p = \Theta(1 / n)$, and showed that the mixing time when starting randomly from any vertex is the order $\bigO(\log n)$ while it was known to be of the order $\bigO(\log^2 n)$ for the mixing time from a worst starting point. Following up, \citet{espuny2022speeding} extended the approach on a larger collection of graphs with small bottlenecks, 
and show applicability of their approach for the
Newman--Watts small-world model and supercritically percolated expander graphs.
In \citep{lovasz1999faster}, the authors introduce the notion of a conductance averaged over subsets of states with various sizes, and demonstrate that it can lead to tighter convergence bounds for counting and volume computation problems.
Finally, the average-mixing time $\atmix$, considered by  \citet{munch2023mixing} and the object of our study, is a measure of convergence of a Markov chain 
defined with respect to the chain being started from its stationary distribution (refer to Section~\ref{section:average-mixing-time}), instead of the worst-case state.

In particular, we emphasize that the notion of average-mixing discussed in this paper differs from that in \citet{peres2015mixing, 10.1214/16-AIHP782, hermon2018characterization, anderson2021mixing}. In these references, averaging is performed by taking an average over two successive steps of the chain to address periodicity issues. In contrast, our analysis involves averaging over the starting state, similar to the definition provided by \cite{munch2023mixing}.

\subsection{Main contributions}
\label{section:main-contributions}
We only introduce the setup and minimal notation in this section. We direct the reader to Section~\ref{section:notation} for the comprehensive set of definitions used in this paper.
Let $X_1, \dots, X_n$ be a time-homogeneous time-discrete $\pi$-stationary ergodic Markov chain with transition operator $P$ over a countable state space $\calX$.
Assume that we are interested in approximating the stationary mean of a real function $f$ by the sample mean computed over the trajectory $X_1, \dots, X_n$. More specifically, we seek a bound on the deviation $\abs{f - \bbE_\pi f}$.
To achieve a bound in the Markovian setting similar to those for iid data, the required sample size is typically multiplied by the worst-case mixing time \citep{chung2012chernoff, paulin2015concentration}.
The underlying idea is that one must allow the chain sufficient time to reach its steady state before regarding the observation as approximately independent and drawn from the stationary distribution.
However, after the chain has reached steady state for the first time, waiting for an additional mixing time addresses the fact that we may have ended up in the worst possible state. While this approach is conservative, it may be unnecessarily cautious. 

\subsubsection{Highlight the significance of the average-mixing time}

Once the chain has reached steady state, since the state we reach should be distributed according to the stationary distribution, it might be preferable to consider the next time to stationarity in a more probabilistic way. 
With this in mind, the question arises whether we can replace the worst-case mixing time with the following more intuitive average version,
\begin{equation*}
    \atmix(\xi) = \argmin_{t \in \bbN} \set{ \sum_{x \in \calX} \pi(x) \tv{ e_x P^t - \pi } \leq \xi },
\end{equation*}
recently analyzed by \citet{munch2023mixing}.
Under this criterion, convergence is achieved faster. In fact convergence can occur arbitrarily faster, for denumerable spaces---in which the worst-case mixing may not even be defined---and even for binary state spaces.
We illustrate how together with an a priori known order of convergence, $\atmix$ can be used to summarize the convergence rate of a Markov chain in lieu of its worst-case mixing time.
Specifically, we show (refer to Lemma~\ref{lemma:large-deviation-bound-average-mixing-time}) that the envelope average-mixing time $\abtmix(\delta)$ can be naturally plugged into the sample complexity 
for controlling the deviation probability from zero of a $1$-bounded centered function $f$.
Indeed, for $\frac{1}{n}\sum_{t=1}^{n} f(X_t) \leq \eps$ to hold with probability at least $1 - \delta$,
we show that 
under sub-exponential $\beta$-mixing of order $b$ (Definition~\ref{definition:beta-mixing-flavors}), it is sufficient to have a trajectory of observations of length
\begin{equation*}
    n \geq  \frac{\left( e/b\right)^{1/b}}{b \eps^2} \abtmix\left( \delta \right) \polylog\left( \frac{1}{\delta}, \frac{1}{\eps} \right),
\end{equation*}
where $\polylog$ is a poly-logarithmic function of its parameters and $e$ is the Euler number. It is instructive to note that $b = 1$ implies exponential $\beta$-mixing, and that the bound continues to hold for $b < 1$.
Similarly, in the polynomial $\beta$-mixing setting of order $b$ (Definition~\ref{definition:beta-mixing-flavors}), we show that for  some natural constant $C \in \bbR_+$,
\begin{equation*}
n \geq C \left( \frac{32}{\eps^2} \log \frac{4}{\delta} \right)^{\frac{b+1}{b}} 2^{1/b} \abtmix\left(\delta\right),
\end{equation*}
is sufficient.
The trajectory lengths listed above both assume a stationary start.
However, this assumption can be relaxed (refer to \appenref{Section~\ref{section:extension-non-stationary}}).

\subsubsection{Estimation of the average-mixing time}

In the case where $\calX$ is a finite space, for $\xi \in (0,1)$ and $\delta \in (0,1)$, 
our chief technical contribution is to show (refer to Corollary~\ref{corollary:average-mixing-time-estimation-finite-space}) that a trajectory $X_1, X_2, \dots, X_n$ of observations of length
    \begin{equation*}
        n \geq \frac{\tmix}{\xi^2} \abs{\calX}^2 \polylog\left(\tmix, \frac{1}{\xi}, \frac{1}{\delta}\right),
    \end{equation*}
    where $\tmix$ is the mixing time of $P$,
is sufficient in order to obtain an estimate $\widehat{\tmix}^{\sharp}(\xi)$ of the average-mixing time $\atmix(\xi)$ which satisfies
\begin{equation*}
\begin{split}
\widehat{\tmix}^{\sharp}(\xi) \in \left[ \tmix^{\sharp}(3\xi/2), \tmix^{\sharp}(\xi/2) \right],
\end{split}
\end{equation*}
with probability at least $1 - \delta$.
    In other words, one can demonstrably estimate with high probability the average-mixing time linearly in the number of parameters of the transition operator, answering a question posed by Luca Zanetti and John Sylvester \citep{zanetti2023}.
    This discovery stands in stark contrast with the worst-case mixing time estimation problem, or the relaxation time estimation problem, for which the sample complexity necessarily depends inversely on the minimum stationary probability \citep{hsu2019mixing, wolfer2024improved} (refer to Section~\ref{section:comparison-relaxation-time-mixing-time-estimation-complexity}).

In fact, contrary to the worst-case mixing time estimation problem,
we show that it is possible to extend the problem and upper bounds on the sufficient trajectory length to denumerable spaces under additional suitable conditions.
Namely, when $\calX \cong \bbN$, we show (refer to Theorem~\ref{theorem:estimation-average-mixing-time}.2) that in the uniformly ergodic setting, a single path of length 
\begin{equation*}
        n \geq \frac{\tmix}{\xi^2}   \calJ_{\infty, \xi/2}  \polylog\left(\tmix, \frac{1}{\xi}, \frac{1}{\delta}\right),
    \end{equation*}
    is sufficient,
where for any $p \in [1, \infty]$, $\calJ_{p, \xi}$ is an entropic term defined as
\begin{equation*}
    \calJ_{p, \xi} \eqdef \sup \set{ \nrm{Q^{(s)}}_{(1 - 1/p)/2}^{1 - 1/p} \colon s \in \bbN, s \leq \atmix(\xi) },
\end{equation*}
and where for $x,x' \in \calX$,  $Q^{(s)}(x,x') = \pi(x) P^s(x,x')$.
We show that one can even analyze and obtain a bound on the sufficient trajectory length in the non uniformly ergodic setting (refer to Theorem~\ref{theorem:estimation-average-mixing-time}.1).

\subsection{Notation and setting}
\label{section:notation}

We let $\calX$ be a countably infinite space of symbols equipped (for simplicity) with a total order. 
Real functions on $\calX$ are written as infinite row vectors and for $x \in \calX$, $e_x$ is the function defined by $x' \in \calX$, $e_x(x') = \pred{x = x'}$, where $\delta[\cdot]$ is the predicate function taking values in $\set{0, 1}$. 
For $p \in (0, \infty)$, and a function $v \in \bbR^{\calX}$,
\begin{equation*}
    \nrm{v}_p \eqdef \left(\sum_{x \in \calX} \abs{v(x)}^p\right)^{1/p},
\end{equation*}
for convenience, we define
\begin{equation*}
    \nrm{v}_0 \eqdef \abs{ \set{ x \in \calX \colon \abs{v(x)} > 0 } },
\end{equation*}
and for any $p \in [0, \infty)$,
\begin{equation*}
    \ell_p \eqdef \set{ v \in \bbR^{\calX} \colon \nrm{v}_p < \infty }.
\end{equation*}
The set of all probability distributions over $\calX$, written $\calP(\calX)$, corresponds to non-negative functions summing to unity.
The dynamics of a time-discrete, time-homogeneous Markov chain
\begin{equation*}
    X_1, X_2, X_3, \dots
\end{equation*}
on $\calX$ are governed by a transition operator represented by an infinite stochastic matrix $P$
$$P(x, x') = \PR{X_{t+1} = x' | X_{t} = x }, \forall x,x' \in \calX.$$
We write $\calW(\calX)$ for the set of all infinite stochastic matrices over $\calX$---and we also identify $\calW(\calX)$ with Markov chains whose transition operator belong to $\calW(\calX)$.
For a finite trajectory of length $n \in \bbN$, started with an initial distribution $\mu \in \calP(\calX)$ and evolving according to $P \in \calW(\calX)$, the probability of observing $(x_1, x_2, \dots, x_n) \in \calX^n$ can be factored as
\begin{equation*}
    \PR[\mu]{X_1 = x_1, \dots, X_n = x_n} = \mu(x_1) \prod_{t=1}^{n-1} P(x_t, x_{t+1}).
\end{equation*}
Matrix multiplication retains its meaning in the countable setting.
Namely, for $\mu \in \calP(\calX)$ and $P \in \calW(\calX)$, we write,
\begin{equation*}
    \mu P = \sum_{x,x' \in \calX} \mu(x) P(x,x') e_{x'}.
\end{equation*}
Extensions of notation to 
$Pf \trn$ with $f \in \ell_1$, or for $P^t$, are similarly defined.
We say that $P$ is ergodic when $P$ is irreducible, aperiodic and positive-recurrent \citep{levin2009markov}.
In that case, there exists a unique stationary distribution $\pi$ verifying $\pi P = \pi$, and we can define the distance to stationarity as
\begin{equation*}
    d(t) \eqdef \sup_{x \in \calX} \tv{e_x P^t - \pi},
\end{equation*}
where $\tv{\cdot}$ is the total variation distance defined for $\mu, \nu \in \calP(\calX)$ by
$\tv{\mu - \nu} \eqdef \frac{1}{2} \nrm{\mu - \nu}_1$.
We say that $P$ is
uniformly ergodic whenever, there exist $M > 0$ and $\rho \in (0,1)$ such that for any $t \in \bbN$,
\begin{equation*}
    d(t) \leq M \rho^t.
\end{equation*}
We write
$\calW^\star(\calX)$ the set of uniformly ergodic Markov chains---identified with their transition operators. 
We recall that when $P \in \calW^\star(\calX)$ the (worst-case) mixing time, defined for any $\xi \in (0, 1)$ as
\begin{equation*}
\label{eq:worst-case-mixing-time}
\begin{split}
    \tmix(\xi) \eqdef \argmin_{t \in \bbN} \set{ d(t) \leq \xi },
\end{split}
\end{equation*}
is finite. It is furthermore customary to define
\begin{equation}
\label{eq:worst-case-mixing-time-fixed}
    \qquad \tmix \eqdef \tmix(1/4),
\end{equation}
as a consequence of the following property \citep[Lemma~4.11]{levin2009markov} 
\begin{equation}
\label{eq:sub-multiplicativity-mixing-time}
    \tmix(\xi) \leq \ceil*{\log_2 \frac{1}{\xi}} \tmix.
\end{equation}
Elements related to the notion of average-mixing time are deferred to the Section~\ref{section:average-mixing-time} for clarity.

\subsection{Outline}
\label{section:outline}
Section~\ref{section:average-mixing-time} introduces the average-mixing time and illustrates its significance.
    Section~\ref{section:estimation} is dedicated to the problem of inferring the $\beta$-mixing coefficients and average-mixing time from a single trajectory of observations.
    In Section~\ref{section:estimation-beta-mixing-coefficients-individually}, we estimate the $\beta$-mixing coefficients of the chain individually, first under general ergodicity condition and then under uniform ergodicity.
    In Section~\ref{section:estimation-average-mixing-time} we deduce upper bounds for the problem of estimating the average-mixing time from a single trajectory. 
    In Section~\ref{section:implications-various-state-space-scales}, we demonstrate that our contribution holds implications at various scales of state spaces and provide more explicit rates under additional structural assumptions on the underlying chain.

    Technical proofs, extensions to non-stationary chains and a discussion on the connection with $\beta$-mixing have been deferred to the appendix.

\section{The average-mixing time}
\label{section:average-mixing-time}
In this section, we first recall the definition of the average-mixing time and then explore some of its basic properties.
For a $\pi$-stationary transition matrix $P \in \calW(\calX)$ and for a proximity parameter $\xi \in (0,1)$, the average-mixing time of $P$ is defined \citep{munch2023mixing} as
\begin{equation*}
\label{eq:average-mixing-time}
\begin{split}
    \atmix(\xi) &\eqdef \argmin_{t \in \bbN} \set{ \beta(t) \leq \xi },
\end{split}
\end{equation*}
where $\beta \colon \bbN \to [0,1]$ is a sequence defined by 
\begin{equation}
\label{definition:stationary-beta-mixing-coefficient}
    \beta(t) \eqdef \sum_{x \in \calX} \pi(x) \tv{ e_x P^t - \pi }.
\end{equation}
Observe that instead of considering the distance to stationarity from the worst-case starting state, as the distance $d$ does, $\beta$ takes an average distance over starting states with respect to the long term probability of being in each state. This choice naturally reduces the contribution of states that are difficult to reach, resulting in a more optimistic convergence criterion.

\begin{remark}
    Similar to the definition of $\tmix$ in 
    \eqref{eq:worst-case-mixing-time-fixed}, one may wish to introduce $\tmix^\sharp \eqdef \tmix^\sharp(1/4)$. However, since $\beta$ does not enjoy the same sub-multiplicativity properties as $d$ ---refer to \eqref{eq:sub-multiplicativity-mixing-time} and  \citep[Lemma~4.11]{levin2009markov}---it is generally not enough to consider an arbitrary value of $\xi \in (0, 1/2)$, and in this manuscript we will keep $\xi$ as a user-fixed value.
\end{remark}

\begin{remark}[Average-mixing time is stationary $\beta$-mixing time]
\label{remark:notation-choice-beta}
The astute reader will have noticed that we choose $\beta$ to denote the average distance to stationarity instead of $d^\sharp$, as in \citep{munch2023mixing}.
This choice is motivated by our observation that the average-mixing time corresponds to the stationary $\beta$-mixing time of the Markov chain, which we discuss in detail in 
\appenref{Section~\ref{section:connection-beta-mixing}}. The sequence $\beta$ defined at \eqref{definition:stationary-beta-mixing-coefficient} will henceforth be referred to as the sequence of (stationary) $\beta$-mixing coefficients.
\end{remark}

An immediate consequence of Remark~\ref{remark:notation-choice-beta} is that the limit $\lim_{t \to \infty} \beta(t) = 0$ always exists \citep{bradley2005basic}, or in other words,  $\atmix(\xi) < \infty$ for any $\xi \in (0,1)$ and any ergodic $P \in \calW(\calX)$ over countable $\calX$.
However, the rate of convergence of the sequence of $\beta$-mixing coefficients can be arbitrarily slow.
\begin{definition}
\label{definition:beta-mixing-flavors}
Rates of convergence for $\beta$-mixing.
\begin{enumerate}
    \item (Sub-exponential, exponential)
    When there exists $\beta_0, \beta_1 \in \bbR_+$ and $b \in (0,1]$ such that for any $s \in \bbN$, $\beta(s) \leq \beta_0 e^{-\beta_1 s^b}$ we say that the chain $\beta$-mixes sub-exponentially when $b \neq 1$, and exponentially when $b = 1$.
    \item (Polynomial) When there exists $\beta_1, b \in \bbR_+$ such that for any $s \in \bbN$, $\beta(s) \leq \beta_1/s^b$, we say that the process $\beta$-mixes algebraically or polynomially.
\end{enumerate}
\end{definition}
Notably, when $P$ is uniformly ergodic with mixing time $\tmix$, it holds that
\begin{equation}
\label{eq:control-beta-mixing-uniformly-ergodic}
    \beta(s) \leq 2 \exp\left( - \frac{\log 2}{\tmix} s\right),
\end{equation}
hence the process is exponentially $\beta$-mixing with $\beta_0 = 2$ and $\beta_1 = \log(2) / \tmix$.
However, nothing a priori precludes the $\beta$-mixing rate from being much faster than \eqref{eq:control-beta-mixing-uniformly-ergodic}.

\subsection{Average-mixing  versus worst-case mixing}

From their respective definitions, it is immediate that $\beta(t) \leq d(t)$. In other words, for any fixed $\xi \in (0,1)$, average-mixing never occurs slower than mixing.
The equality $\tmix = \atmix$ is notably achieved
for transitive chains \citep[Remark~1]{munch2023mixing}, which are chains such that for any $x_0,x_0' \in \calX$, there exists a bijection $\phi \colon \calX \to \calX$, with $\phi(x_0) = x_0'$ and  for all $x,x' \in \calX, P(\phi(x), \phi(x')) = P(x,x')$ \citep[Section~2.6.2]{levin2009markov}.
Transitive chains include, but are not limited to, random walks on groups \citep{mckay1996vertex}.
For non-transitive chains, the mixing time can be dramatically larger than the average-mixing time, even when the state space is small, as made precise below.

\begin{proposition}
Let $M \in \bbR_+$ be arbitrarily large and any $\xi \in (0,1)$. There exists a transition operator $P$ such that
\begin{equation*}
    \tmix(\xi)  > M \tmix^{\sharp}(\xi).
\end{equation*}
\end{proposition}
\begin{proof}
    Lemma~\ref{lemma:arbitrary-gap-between-average-and-worst-case-mixing-time} establishes this fact in the case where $\calX = \set{0,1}$.
\end{proof}

In fact, as we will see in the next section, a regime of particular interest to us is when the chain $\beta$-mixes much faster than $\Theta(\tmix)$.
We begin by recalling the well-studied class of birth and death Markov chains, which will serve as our running example.
\begin{example}[Birth and death Markov chains]
\label{example:birth-and-death}
 We let $\calX = \bbN$, and for three sequences $u,v,w \in [0,1]^\bbN$ such that $u(1) = 0$ and $\forall n \in \bbN, u(n) + v(n) + w(n) = 1$, we define the transition operator,
\begin{equation*}
P_{u,v,w} = \begin{pmatrix}
    v(1) & w(1) & 0 & 0 & 0 & \hdots \\
    u(2) & v(2) & w(2) & 0 & 0 & \vdots \\
    0 & u(3) & v(3) & w(3) & 0 & \ddots \\
    0 & 0 & u(4) & v(4) & \ddots & \ddots \\
    \vdots & \hdots & \ddots & \ddots & \dots & \ddots \\
\end{pmatrix}.
\end{equation*}
Due to their ubiquity, ergodicity properties of birth and death chains have been extensively studied in the literature \citep{miclo1999example, ding2010total, kovchegov2010, chen2013mixing}.
In the special case where $v(1) = 0$ and for $n \geq 2$ the sequences $u,v,w$ are constant with $u > w$ and $v > 0$, it can be shown that the family of chains
\begin{equation*}
P_{u,v,w} = \begin{pmatrix}
    0 & 1 & 0 & 0 & 0 & \hdots \\
    u & v & w & 0 & 0 & \vdots \\
    0 & u & v & w & 0 & \ddots \\
    0 & 0 & u & v & \ddots & \ddots \\
    \vdots & \hdots & \ddots & \ddots & \dots & \ddots \\
\end{pmatrix}
\end{equation*}
is geometrically ergodic \citep{meitz2021subgeometric, kovchegov2010},
that is, there exists $M \colon \calX \to \bbR_+$ and $\rho \in (0,1)$ such that for any $t \in \bbN$ and any $n_0 \in \calX$,
\begin{equation}
\label{eq:geometric-ergodicity}
    \tv{e_{n_0} P_{u,v,w}^t - \pi} \leq M(n_0) \rho^t.
\end{equation}
In particular, the geometric rate of the above family is given by
\begin{equation*}
    \rho = \max \set{ v + 2\sqrt{uw}, \frac{u}{u + v} }.
\end{equation*}
While geometric ergodicity is weaker than uniform ergodicity in as much as the distance to the stationary distribution depends on the starting state, convergence remains exponential and leads to a powerful convergence analysis based on spectral methods (refer also to Section~\ref{section:spectral-analysis}). But we now turn our attention to the even less favorable case of sub-geometric convergence, where the aforementioned spectral methods break down. This more challenging setting can arise even in for the seemingly simple class of birth and death chains in this example.
Indeed, let us borrow the following instantiation of a Chebyshev--type random walk, proposed by 
\citet[Section~4]{kovchegov2013class}.
Let $\theta > 0$ and $\lambda \geq \frac{2\theta^2}{(1 + \theta)(1 + 3\theta)}$. We define
\begin{equation*}
    \begin{split}
        u(1) = 0, \qquad v(1) &= 1 - w(1), \qquad w(1) = \frac{1}{(1 + \lambda)(\theta + 1)}\\
\end{split}
\end{equation*}
while for $n \geq 2$,
\begin{equation*}
    \begin{split}
        u(n) = \frac{1}{2(1+\lambda)} \cdot \frac{1 + (2n - 1) \theta}{1 + (2n - 3) \theta}, \qquad w(n) = \frac{1}{2(1+\lambda)} \cdot \frac{1 + (2n - 3) \theta}{1 + (2n - 1) \theta}, \\
    \end{split}
\end{equation*}
and $v(n) = 1 - u(n) - w(n)$. For convenience, we denote $P_{\theta, \lambda} = P_{u,v,w}$ for the above-defined sequences.
It can be verified that the stationary distribution of $P_{\theta, \lambda}$ satisfies,
\begin{equation*}
    \pi_\theta(1) = \frac{\theta}{1  + \theta},\qquad \pi_{\theta}(n) = \frac{2\theta}{(1 + (2n - 1) \theta)(1 + (2n - 3) \theta)}, \text{ when } n \geq 2.
\end{equation*}
One interesting feature is that the chain does not have a spectral gap---or in other words, it ``[lies] outside the scope of geometric ergodicity theory'' \citep{kovchegov2013class}.
In fact, for the random walk originating at state $n_0 \in \bbN$, it holds \citep[Theorem~2]{kovchegov2013class} that for any $t \in \bbN$,
\begin{equation*}
    \frac{c}{\sqrt{t}} \leq \tv{e_{n_0}P_{\theta, \lambda}^t - \pi_\theta} \leq  \frac{C n_0 \log(t + n_0 + 2)}{\sqrt{t}},
\end{equation*}
where $c$ is a constant depending on $\theta$, $\lambda$ and the starting state $n_0$, while $C$ is a constant which depends only on $\theta$ and $\lambda$.
Let $\xi \in (0,1)$ and denote 
$$n_\star = n_\star(\theta, \xi/2) = \min_{n \in \bbN} \set{ \sum_{k = n + 1}^{\infty} \pi_\theta(k) < \xi /2 }.$$
We can thus bound the $\beta$-mixing coefficient as follows,
\begin{equation*}
\begin{split}
    \beta(s) &\leq \xi/2 + \sum_{n = 1}^{n_{\star}} \pi_\theta(n) \tv{ e_{n} P_{\theta, \lambda}^s - \pi_\theta} \leq \xi/2 +  C' \frac{ \log\left(s + 2 + n_\star\right)}{\sqrt{s}} H_{n_\star}, \\
\end{split}
\end{equation*}
where $C'$ is a constant which depends only on $\theta$ and $\lambda$ and for any $k \in \bbN$, $H_{k}$ denotes the $k$th harmonic number which is bounded from above as $H_k \leq 1 + \log(k)$.
Furthermore, one can verify that $n_\star \leq 2 + \frac{1}{\theta \xi}$.
As a result, there exists a constant $C''$ which depends only on $\theta$ and $\lambda$ such that,
\begin{equation*}
\begin{split}
    \beta(s) &\leq \xi/2 + C'' \frac{\log^2\left(s + 4 + 1/(\theta \xi)\right)}{\sqrt{s}}, \\
\end{split}
\end{equation*}
and we obtain that
\begin{equation*}
    \atmix(\xi) \leq \frac{1}{\xi^2} \polylog(1/\xi),
\end{equation*}
where $\polylog$ is a poly-logarithmic function of its argument, which may also depend on $\theta$ and $\lambda$.
Our example demonstrates how the contributions to $\beta(s)$ starting from the tail of the stationary distribution tend to be largely absorbed.
\end{example}

\subsection{Deviation of additive functionals of Markov chains}
\label{section:bounding-deviations}
We now illustrate how the average-mixing time, through its connection to $\beta$-mixing naturally appears as a quantity of interest when bounding deviations of functions evaluated on a trajectory sampled from a Markov chain.
More formally, let $X_1, \dots, X_n$ be a Markov chain over a countable space $\calX$, and whose dynamics are governed by $P \in \calW(\calX)$ with stationary distribution $\pi$.  
For simplicity, we consider a real bounded function $f$ on $\calX$, which is centered in the sense where its stationary expected value vanishes, that is $\bbE_\pi f = 0$.
The ergodic theorem states that the empirical mean taken on an infinite trajectory will similarly vanish,
\begin{equation*}
    \lim_{n \to \infty} \frac{1}{n} \sum_{t = 1}^n f(X_t) = 0.
\end{equation*}
In practice however, we are more likely to be interested in controlling the deviation of the sample mean from its expected value non-asymptotically. 
Given an upper envelope $\bar{\beta}$ on the sequence of coefficients of $P$, we define the envelope average mixing time by $\abtmix(\xi) = \inf \{s \in \bbN \colon \bar{\beta}(s) \leq \xi\}$.
The following lemma illustrates how to recover a finite sample bound on the deviation probability, given a bound on the average-mixing time induced by $\bar{\beta}$ and the rate of
convergence of the sequence of $\beta$-mixing coefficients.

\begin{lemma}[Bounding deviation of additive functionals evaluated on Markov chains]
\label{lemma:large-deviation-bound-average-mixing-time}
Let $f \colon \calX \to [-1, 1]$ with $\bbE_\pi f = 0$, let $\eps, \delta \in (0,1)$.
For any $\pi$-stationary ergodic Markov chain $X_1, X_2, \dots, X_n$ with transition operator $P \in \calW(\calX)$,
and envelope average-mixing time $\abtmix \colon (0,1) \to \bbN$,
the following statements hold.
\begin{enumerate}
    \item (Sub-exponential) If there exist $\beta_0, \beta_1, b \in \bbR_+$ with $\beta_0 \geq 1$, $b \in (0,1]$,
    such that for any $s \in \bbN$,
    $\beta(s) \leq \beta_0 \exp(-\beta_1 s^b)$, then there exists a universal constant $C \in \bbR_+$
such that for 
\begin{equation*}
    n \geq \frac{C}{\eps^2}  \left( \frac{e}{b}\right)^{1/b} \abtmix\left( \xi(\eps, \delta) \right) \log \left( \frac{4}{\delta}\right),
\end{equation*}
with 
$$\xi(\eps, \delta) \eqdef \frac{\delta \eps^2}{64 \log \left(\frac{4}{\delta}\right)},$$
it holds that $\frac{1}{n}\sum_{t=1}^{n} f(X_t) \leq \eps$ with probability at least $1 - \delta$.
The above upper on the path length can be further streamlined as
\begin{equation*}
\begin{split}
    n
    &\geq
    \frac{C}{\eps^2}
    \frac{1}{b}
    \left(\frac{e}{b}\right)^{1/b}
    \abtmix\left(\delta/e\right)
    \log\left(\frac{4}{\delta}\right)
    \left[
        1+
        \log\left(
            \frac{64\log(4/\delta)}
                 {e\eps^2}
        \right)
    \right]^{1/b} \\
    &= \frac{1}{\eps^2} \frac{1}{b}
    \left(\frac{e}{b}\right)^{1/b}
    \abtmix\left(\delta/e\right)
    \polylog_b(1/\delta, 1/\eps).
\end{split}
\end{equation*}
for some universal constant $C \in \bbR_+$.
\item (Polynomial) If there exist $\beta_1, b \in \bbR_+$ with $\beta_1 \geq 1$, such that for any $s \in \bbN$, $\beta(s) \leq \frac{\beta_1}{s^b}$, then there exists a universal constant $C \in \bbR_+$
such that for 
\begin{equation*}
n \geq C \left( \frac{32}{\eps^2} \log \frac{4}{\delta} \right)^{\frac{b+1}{b}} 2^{\frac{1}{b}} \abtmix\left(\delta\right),
\end{equation*}
it holds that $\frac{1}{n}\sum_{t=1}^{n} f(X_t) \leq \eps$ with probability at least $1 - \delta$.
\end{enumerate}
\end{lemma}

\begin{proof}[Proof sketch]

We adopt a traditional method described in \citet{yu1994rates}.
This approach combines a blocking technique, credited to \citet{bernstein1927extension} in the context of $m$-dependent processes, together with an application of Hoeffding's inequality to control the deviation probability of a sum of independent observations. We further optimize for the block size and reformulate the results to obtain an expression involving the envelope average-mixing time under various rates of $\beta$-mixing. 
Refer to the supplementary material for the complete proof.
\end{proof}

As a consequence of Lemma~\ref{lemma:large-deviation-bound-average-mixing-time}, in both the sub-exponential and polynomial settings, provided an assumption on the convergence order $b$, estimates on the envelope average-mixing time to proximity $\Theta(\delta)$ yield corresponding deviation bounds. 
In the uniformly ergodic setting, which implies exponential
$\beta$-mixing with $b=1$, we obtain a deviation inequality involving the envelope average-mixing time in lieu of the mixing time \citep[Theorem~3.1]{chung2012chernoff}, \citep[Corollary 2.10]{paulin2015concentration}.
The lemma generalizes to non-centered functions, lower deviations and arbitrarily bounded functions. Our formulation above serves as motivation for analyzing and inferring the average-mixing time.

\begin{remark}
The astute reader will observe that, in the uniformly ergodic setting, the logarithmic dependence in the confidence parameter $\delta$ is slightly sharper for bounds involving the worst-case mixing time.
    \citep[Theorem~3.1]{chung2012chernoff}, \citep[Corollary 2.10]{paulin2015concentration}. This is a consequence of our proof technique relying on $\beta$-mixing.
\end{remark}

\subsection{Under geometric ergodicity}

In this section, we briefly examine the average-mixing time under geometric ergodicity, first using spectral methods under time-reversibility, and then through more broadly applicable Lyapunov-type methods.

\subsubsection{Spectral methods under reversibility}
\label{section:spectral-analysis}
Spectral methods are known to provide a direct and concise framework for analyzing the worst-case mixing time.
In this section, we briefly examine their applicability to the study of the average-mixing time.
We endow $\bbR^{\calX}$ with the inner product $\langle f,g \rangle_{\pi} \eqdef \sum_{x \in \calX} f(x)g(x)\pi(x)$,
where $\pi$ is a positive distribution.
We regard $P$ as a linear operator on the resulting Hilbert space $$\ell_2(\pi) \eqdef
\left(
\set{f \in \bbR^{\calX} :
\sum_{x \in \calX} \pi(x) f(x)^2 < \infty},
\langle \cdot, \cdot \rangle_{\pi}
\right),$$ and further assume $P$ to be geometrically ergodic---refer to \eqref{eq:geometric-ergodicity}.
Recall that $P \in \calW^\star(\calX)$ is called $\pi$-reversible when the following detailed-balance equation holds,
\begin{equation*}
    \pi(x)P(x,x') = \pi(x')P(x',x), \qquad \forall x,x' \in \calX,
\end{equation*}
which is equivalent to stating that $P$ is self-adjoint in $\ell_2(\pi)$.
Under reversibility, the spectrum of $P$, denoted $\Spec(P)$, is contained in the real line. The absolute spectral gap of $P$ is defined as
\begin{equation}
\label{definition:absolute-spectral-gap}
    \gamma_\star \eqdef 1 - \sup \set{ \abs{\lambda} \colon \lambda \in \Spec(P), \lambda \neq 1 }.
\end{equation}
Geometric ergodicity and existence of a positive spectral gap are known to be equivalent \citep{kontoyiannis2012geometric}.
We first establish that, at least under mild assumptions on $P$, the average-mixing time can be controlled by the relaxation time $\trel \eqdef 1/\gamma_\star$, even over countably infinite spaces.

\begin{lemma}
\label{lemma:spectral-upper-bound}
    Let $P \in \calW(\calX)$ be $\pi$-stationary for $\nrm{\pi}_{1/2} < \infty$, geometrically ergodic, reversible and let $\xi \in (0, 1/2)$.
    When $P$ regarded as the linear operator $P \colon \ell_2(\pi) \to \ell_2(\pi)$
    is compact\footnote{For instance, being trace-class is a sufficient condition for $P$ to be compact.}, it holds that
    \begin{equation*}
    \atmix(\xi) \leq \ceil*{\trel \log \frac{\nrm{\pi}_{1/2}}{2 \xi}},
    \end{equation*}
    where $\trel$ is the relaxation time of $P$ defined by $\trel = 1/\gamma_\star$ with $\gamma_\star$ the absolute spectral gap of $P$ defined in \eqref{definition:absolute-spectral-gap}.
\end{lemma}

The half quasi-norm of the stationary distribution $\inrm{\pi}_{1/2}$, making its first appearance in our study, is a natural measure of complexity of the stationary distribution $\pi$, directly related to its R\'{e}nyi entropy of order $1/2$. 
Note that $\inrm{\pi}_{1/2}$ could be infinite, in which case Lemma~\ref{lemma:spectral-upper-bound} becomes vacuous.
However, Lemma~\ref{lemma:spectral-upper-bound} compares favorably with the bound of \citet[Theorem~5, Corollary~2]{munch2023mixing}, which is inversely proportional to the minimum (non-zero) transition probability $P_\star \eqdef \inf_{\substack{x,x' \in \calX \\ P(x,x') > 0}} P(x,x')$.
Additionally, Lemma~\ref{lemma:spectral-upper-bound} can be instructively compared with \citet[Theorem~12.4]{levin2009markov} for the mixing time $\tmix(\xi)$,
\begin{equation*}
    \tmix(\xi) \leq  \trel \log \frac{1}{\xi \pi_\star},
\end{equation*}
where $\pi_\star \eqdef \min_{x \in \calX} \pi(x)$,
which is only informative when $\abs{\calX} < \infty$.

\subsubsection{Lyapunov methods}
\label{section:lyapunov-analysis}
Following a classical Banach space structure, for $V\colon \calX \to [1, \infty)$,
the definitions of the $V$-norm of a function $f$ on $\calX$ and a signed measure $\mu$ on $\calX$ are given \citep[Definition~D.3.1]{douc2018markov} respectively as ,
\begin{equation*}
    \nrm{f}_V \eqdef \sup_{x \in \calX} \frac{\abs{f(x)}}{V(x)} \qquad \nrm{\mu}_V \eqdef \sum_{x \in \calX} \abs{\mu(x)} V(x) = \sup_{\nrm{f}_V \leq 1} \abs{\sum_{x \in \calX} f(x) \mu(x)}.
\end{equation*}
We say that $P$ satisfies $V$-geometrical ergodicity \citep{herve2020v} when there exists $K \geq 1$ and $\rho_V \in (0,1)$ such that for all $x \in \calX$ and all $t \in \bbN$,
\begin{equation*}
    \nrm{e_x P^t - \pi}_V \leq K V(x) \rho_V^t.
\end{equation*}
In particular, $P$ is known to be $V$-geometrically ergodic under irreducibility, aperiodicity, drift and minorization conditions
\citep{herve2020v}.
In this case, observe that we can bound the stationary $\beta$-coefficient as follows,
\begin{equation*}
\begin{split}
\beta(t) &= \sum_{x \in \calX} \pi(x) \tv{e_xP^t - \pi} \\
&= \frac{1}{2} \sum_{x \in \calX} \pi(x) \sum_{x' \in \calX} \abs{P^t(x,x') - \pi(x')} \\
&\leq \frac{1}{2} \sum_{x \in \calX} \pi(x) \sum_{x' \in \calX} \abs{P^t(x,x') - \pi(x')} V(x') \\
&\leq \frac{K}{2} \sum_{x \in \calX} \pi(x) V(x) \rho_V^t \\
&= \frac{K}{2} \rho_V^t \nrm{\pi}_V. \\
\end{split}
\end{equation*}
As a result,
\begin{equation*}
    \atmix(\xi) \leq  1 \lor \Bigg\lceil \cfrac{\log  \nrm{\pi}_V + \log \frac{K}{2\xi} }{\log \frac{1}{\rho_V}} \Bigg\rceil,
\end{equation*}
which is finite whenever $\nrm{\pi}_V = \E[\pi]{V}$ is finite.

\section{Estimation of average convergence from a single trajectory}
\label{section:estimation}
To apply the results in Section~\ref{section:average-mixing-time}, one needs to have an upper bound on the average-mixing time. However, such bound may not be known a priori, and may need to be estimated from the data.
In this section, we consider the problem of estimating the average-mixing time $\atmix$ from a single trajectory of observations. Specifically, for a user-fixed proximity parameter $\xi \in (0,1)$ and a confidence parameter $\delta \in (0,1)$, our goal is to construct an interval $I_\delta \subset [1, \infty)$, such that given a trajectory $X_1, \dots, X_n$ sampled from an unknown transition matrix $P$, and for $n$ sufficiently large, it holds that $\tmix^{\sharp}(\xi) \in I_\delta$ with probability at least $1 - \delta$.

We first analyze the intermediary problem of estimating the $\beta$-mixing coefficients of the process.
We subsequently show how to convert such results into estimation procedures for the average-mixing time both in the uniformly ergodic and non-uniformly ergodic settings
(refer to Figure~\ref{fig:estimation-structure}).

\begin{figure}
    \centering
    \tikzset{every picture/.style={line width=0.75pt}} %

\begin{tikzpicture}[x=0.65pt,y=0.65pt,yscale=-1,xscale=1]

\draw  [color=wred  ,draw opacity=1 ][fill=wred  ,fill opacity=0.17 ] (321,270.33) -- (464,270.33) -- (464,340.67) -- (321,340.67) -- cycle ;
\draw  [color=wred  ,draw opacity=1 ] (92,121.67) -- (464,121.67) -- (464,185.33) -- (92,185.33) -- cycle ;
\draw  [color=wred  ,draw opacity=1 ][fill=wred  ,fill opacity=0.06 ] (92,185.33) -- (464,185.33) -- (464,270.33) -- (92,270.33) -- cycle ;
\draw [color={rgb, 255:red, 155; green, 155; blue, 155 }  ,draw opacity=1 ] [dash pattern={on 0.84pt off 2.51pt}]  (206,100.33) -- (206,270.33) ;
\draw [color=wred  ,draw opacity=0.47 ]   (321,69.67) -- (321,269.67) ;

\draw (500,279.67) node [anchor=north west][inner sep=0.75pt]  [rotate=-45] [align=left] {Finite $\calX$};
\draw (508,195) node [anchor=north west][inner sep=0.75pt]  [rotate=-45] [align=left] {Uniform\\ergodicity};
\draw (500,126) node [anchor=north west][inner sep=0.75pt]  [rotate=-45] [align=left] {Ergodicity};
\draw (190,75) node [anchor=north west][inner sep=0.75pt]   [align=left] {$\beta(s)$};
\draw (377,75) node [anchor=north west][inner sep=0.75pt]   [align=left] {$\atmix(\xi)$};
\draw (127,101) node [anchor=north west][inner sep=0.75pt]   [align=left] {MAD};
\draw (247,102) node [anchor=north west][inner sep=0.75pt]   [align=left] {PAC};
\draw (124,145) node [anchor=north west][inner sep=0.75pt]   [align=left] {Th.~\ref{theorem:estimation-beta-coefficients-skip-wise-without-tmix}};
\draw (235,145) node [anchor=north west][inner sep=0.75pt]   [align=left] {Th.~\ref{theorem:estimation-beta-coefficients-general-ergodicity-skip-wise}};
\draw (355,145) node [anchor=north west][inner sep=0.75pt]   [align=left] {Th.~\ref{theorem:estimation-average-mixing-time}.1};
\draw (235,220) node [anchor=north west][inner sep=0.75pt]   [align=left] {Th.~\ref{theorem:estimation-beta-coefficients-uniform-ergodicity-skip-wise-pac}};
\draw (355,220) node [anchor=north west][inner sep=0.75pt]   [align=left] {Th.~\ref{theorem:estimation-average-mixing-time}.2};
\draw (124,220) node [anchor=north west][inner sep=0.75pt]   [align=left] {Th.~\ref{theorem:estimation-beta-coefficients-uniform-ergodicity-skip-wise-amd}};
\draw (365,296) node [anchor=north west][inner sep=0.75pt]   [align=left] {Corr. \ref{corollary:average-mixing-time-estimation-finite-space}};
\draw (190,150) node [anchor=north west][inner sep=0.75pt]  [align=left] {$\Longrightarrow$};
\draw (395,255) node [anchor=north west][inner sep=0.75pt]  [rotate=-90] [align=left] {$\Longrightarrow$};
\draw (308,150) node [anchor=north west][inner sep=0.75pt]   [align=left] {$\Longrightarrow$};
\draw (308,225) node [anchor=north west][inner sep=0.75pt]   [align=left] {$\Longrightarrow$};
\draw (190,225) node [anchor=north west][inner sep=0.75pt]   [align=left] {$\Longrightarrow$};

\end{tikzpicture}
    \caption{Logical flow between the estimation results of Section~\ref{section:estimation}. MAD: Mean Absolute Deviation; PAC: Probably Approximately Correct.}
    \label{fig:estimation-structure}
\end{figure}

\subsection{Estimation of \texorpdfstring{$\beta$}{・趣ｽｲ}-mixing coefficients}
\label{section:estimation-beta-mixing-coefficients-individually}

We begin by analyzing a sequence of plug-in estimators for the $\beta$-mixing coefficients.
Our analysis relies on stationary skipped trajectories, where for $s \in \bbN, s < n$, we denote
\begin{equation*}
\begin{split}
X^{(s)} \eqdef X_{1}, X_{1 + s}, X_{1 + 2s}, \dots, X_{1 + \floor{(n-1)/s}s},
\end{split}
\end{equation*}
with transition operator $P^s$ and initial distribution $\pi$.
It will be convenient to define the following counting random variables,
\begin{equation*}
\begin{split}
N_{x}^{(s)} &\eqdef \sum_{t=1}^{\floor{(n-1)/s}} \pred{X_{1 + s(t-1)} = x}, 
\\ N_{xx'}^{(s)} &\eqdef \sum_{t=1}^{\floor{(n-1)/s}} \pred{X_{1 + s(t-1)} = x, X_{1 + st} = x'}.
\end{split}
\end{equation*}
We use the shortcut notation $N_x = N_x^{(1)}$ and $N_{xx'} = N_{xx'}^{(1)}$ to respectively denote number of visits to state $x$ and transitions from $x$ to $x'$ on the original chain.
For $s <n$, we define the following estimator $\widehat{\beta}(s) \colon \calX^n \to [0,1]$ for $\beta(s)$,
\begin{equation}
\label{eq:estimator-beta-s}
    \widehat{\beta}(s) \eqdef \frac{1}{2 \floor{(n-1)/s}} \sum_{x, x' \in \calX}  \abs{N_{x x'}^{(s)} -  \frac{N^{(s)}_x N^{(s)}_{x'}}{\floor{(n-1)/s}}},
\end{equation}
and by convention, we set $\widehat{\beta}(s) \eqdef 0$ for $s \geq n$.

\subsubsection{Under general ergodicity condition}
\label{section:estimation-beta-mixing}
Under general---possibly non uniform---ergodicity, the mixing time may be infinite, and the $\beta$-mixing coefficients may decay at arbitrary slow rate.
We first analyze the convergence properties of our estimator $\widehat{\beta}(s)$ under an arbitrary $\beta$-mixing rate,
 and will subsequently focus on exponential, sub-exponential and polynomial rates of convergence (refer to Definition~\ref{definition:beta-mixing-flavors}).

\begin{theorem}[Mean Absolute Deviation]
    \label{theorem:estimation-beta-coefficients-skip-wise-without-tmix}
Let $s, n \in \bbN$ with $s < n$, let $p \in \bbR_+$, with $p \geq 1$, and
let $\widehat{\beta}(s) \colon \calX^n \to [0,1]$ be the estimator  defined in \eqref{eq:estimator-beta-s}.
 For any $\pi$-stationary ergodic Markov chain $X_1, X_2, \dots, X_n$ with transition operator $P \in \calW(\calX)$, it holds that
 \begin{equation*}
\begin{split}
    \bbE_\pi \abs{ \widehat{\beta}(s) - \beta(s) } 
    \leq 3 \sqrt{\frac{ \left(1/2 + \calB_p^{(s)}\right) \calJ_p^{(s)}}{\floor{(n-1)/s}}  },
\end{split}
\end{equation*}
with
\begin{equation*}
    \calB_{p}^{(s)} \eqdef \sum_{t = 0}^{\floor{(n-1)/s} - 1} \beta(st)^{1/p}, \qquad \calJ_p^{(s)} \eqdef \nrm{Q^{(s)}}^{1 - 1/p}_{(1 - 1/p)/2},
\end{equation*}
where for $s \in \bbN$, $\beta(s)$ follows \eqref{definition:stationary-beta-mixing-coefficient}, for $x,x' \in \calX$, we wrote $Q^{(s)}(x,x') = \pi(x) P^s(x,x')$, and for $q \in \bbR_+$,
\begin{equation*}
    \nrm{Q^{(s)}}^{q}_{q} \eqdef \sum_{(x,x')\in \calX^2} Q^{(s)}(x,x')^q.
\end{equation*}
\end{theorem}

\begin{theorem}[$(\eps, \delta)$-PAC Bound]
\label{theorem:estimation-beta-coefficients-general-ergodicity-skip-wise}
Let $\eps, \delta \in (0,1)$, $s, n \in \bbN$ with $n > 2s + 1$, $p \in \bbR_+$, with $p \geq 1$ and let $\widehat{\beta}(s) \colon \calX^n \to [0,1]$ be the estimator  defined in \eqref{eq:estimator-beta-s}.
There exists a universal constant $\Ce \in \bbR_+$ such that for any ergodic stationary Markov chain $X_1, X_2, \dots, X_n$ with transition operator $P \in \calW(\calX)$, when
\begin{equation*}
    n \geq 1 + \Ce \frac{\log(8/\delta)}{\eps^2} \max \set{ \frac{s}{\eps^2} \calB_p^{(s)} \calJ^{(s)}_p, s+ \atmix\left(\xi(\eps, \delta) \right) }, \qquad \xi(\eps,\delta) = \frac{ \eps^2 \delta}{\Ce \log(8/\delta)},
\end{equation*}
    with probability at least $1 - \delta$, it holds that
    \begin{equation*}
        \abs{\widehat{\beta}(s) - \beta(s)} \leq \eps.
    \end{equation*}
\end{theorem}

In Theorem~\ref{theorem:estimation-beta-coefficients-skip-wise-without-tmix} and Theorem~\ref{theorem:estimation-beta-coefficients-general-ergodicity-skip-wise}, we observe that the two quantities $\calB_{p}^{(s)}$ and $\calJ_p^{(s)}$, discussed below, control the estimation rate.
In infinitely denumerable state spaces, for any $s \in \bbN$,
the well-definition of the estimation rate 
therefore hinges on the existence of a $p \in [1, \infty)$ such that both $\calB_{p}^{(s)}$ and $\calJ_p^{(s)}$ are finite.

\textbf{The $\beta$-mixing parameter $\calB_p^{(s)}$. } This parameter summarizes the $\beta$-mixing properties of the process and is controlled from above by the $\ell_{1/p}$ pseudo-norm of the sequence of the $\beta^{(s)}$-mixing coefficients,
\begin{equation*}
    \calB_p^{(s)} \leq \nrm{\beta^{(s)}}_{1/p}^{1/p}.
\end{equation*}
It will be convenient to introduce the shorthand notation $\calB_p \eqdef \calB_p^{(1)}$ to refer to the sum of the $\beta$-mixing coefficients.
It is instructive to observe that since $\beta(s) \leq 1$ for any $s \in \bbN$, $\calB_{p}^{(s)}$ increases together with $p$, and
\begin{equation*}
    \lim_{p \to \infty} \calB_{p}^{(s)} = \sum_{t=0}^{\floor{(n-1)/s}-1}
    1\left[\beta(st)>0\right]
    \leq  \frac{n}{s}.
\end{equation*}
Additionally, 
\begin{equation*}
    r \leq s \implies \calB_p^{(r)} \geq \calB_p^{(s)}.
\end{equation*}
In the following Lemma~\ref{lemma:bound-mixing-factor}, we readily provide more explicit upper bounds on $\calB_{p}^{(s)}$ under additional assumptions on the mixing rate.

\begin{lemma}
\label{lemma:bound-mixing-factor}
Let $n, s \in \bbN$, with $s < n$, let $p \in \bbR_+, p \geq 1$.
 For any $\pi$-stationary ergodic 
Markov chain $X_1, X_2, \dots, X_n$, the following statements hold.
\begin{enumerate}
    \item (Exponential) If for any $t \in \bbN \cup \{0\}, \beta(t) \leq \beta_0 \exp(-\beta_1 t)$, with $\beta_0, \beta_1 \in \bbR_+$,
    \begin{equation*}
\begin{split}
    \calB_p^{(s)} &\leq  \frac{\beta_0^{1/p}}{1 - \exp(- \beta_1 s / p)}. \\
\end{split}
\end{equation*}
\item (Sub-exponential) If for any $t \in \bbN \cup \{0\}$, $\beta(t) \leq \beta_0 \exp(-\beta_1 t^b)$ , with $\beta_0, \beta_1 \in \bbR_+$, $b \in (0,1)$,
    \begin{equation*}
\begin{split}
    \calB_p^{(s)} &\leq \beta_0^{1/p}\left( 1 + \exp(-\beta_1s^b/p) + \frac{\Gamma\left(\frac{1}{b}, s^b \beta_1  / p\right)}{b s (\beta_1  / p)^{1/b}} \right),
\end{split}
\end{equation*}
where $\Gamma$ is the (upper) incomplete Euler gamma function, given by $\Gamma (a,x) = \int _{x}^{\infty}t^{a-1} e^{-t} dt$.
\item (Polynomial) If for any $t \in \bbN$, $\beta(t) \leq \beta_1/t^b$ , with $\beta_1 \in \bbR_+$, $b \in (1, \infty)$, $b/p \in (1, \infty)$, and $\beta(0) = \beta_0 \in \bbR_+$,
\begin{equation*}
\begin{split}
    \calB_p^{(s)} &\leq \beta_0^{1/p} + \zeta(b/p) \frac{\beta_1^{1/p}}{s^{b/p}},
\end{split}
\end{equation*}
where $\zeta$ is the Riemann zeta function defined by $\zeta (r)=\sum _{t=1}^{\infty }t^{-r}$.
\end{enumerate}
\end{lemma}

\textbf{The entropic parameter $\calJ_p^{(s)}$.}
For $q < 1$, since the $\ell_{q}$ quasi-norm and the $q$th order R\'{e}nyi entropy $H_{q}$ of a distribution $\mu \in \calP(\calX)$
are related by the expression
\begin{equation*}
    \frac{1}{1-q}\log \nrm{\mu}_{q}^{q} = H_{q}(\mu),
\end{equation*}
the quantity $\calJ_{p}^{(s)}$ pertains to an entropic property inherent in the Markov process.
Observe that in contradistinction with $\calB_{p}^{(s)}$, the quantity $\calJ_{p}^{(s)}$ decreases with $p$, leading to a trade-off in the choice of the parameter $p$. A necessary condition on $P$ for a rate to exist is readily given by
\begin{equation}
\label{definition:J-infty}
   \calJ_\infty^{(s)} \eqdef \nrm{Q^{(s)}}_{1/2} < \infty,
\end{equation}
that is, the tail of the distribution of stationary pairs of observations separated by a time $s$ must not decay more slowly than an inverse square polynomial.
A more in-depth analysis of $\calJ_p^{(s)}$ is deferred to Section~\ref{section:implications-various-state-space-scales}, where we show how to obtain upper bounds under natural assumptions.

\subsubsection{Under uniform ergodicity}
\label{section:estimation-beta-mixing-uniform-ergodicity}
Under uniform ergodicity, the mixing time is finite and the $\beta$-mixing coefficients decay exponentially. In this setting, we show how to obtain a stronger sample complexity upper bound for high-probability estimates on the individual $\beta$-mixing coefficients. In particular, uniform ergodicity removes the need to consider the trade-off between mixing and entropic parameters mentioned in the previous section.

\begin{theorem}[Mean Absolute Deviation]
\label{theorem:estimation-beta-coefficients-uniform-ergodicity-skip-wise-amd}
Let $s, n \in \bbN$ with $n > 2s + 1$, and let $\widehat{\beta}(s) \colon \calX^n \to [0,1]$ be the estimator  defined in \eqref{eq:estimator-beta-s}.
There exists a universal constant $\Cue \leq 3\sqrt{2}$ such that
for any $\pi$-stationary uniformly ergodic Markov chain $X_1, X_2, \dots, X_n$ with transition operator $P \in \calW^\star(\calX)$ such that
$\calJ_{\infty}^{(s)}<\infty$,
\begin{equation*}
\begin{split}
    \bbE_\pi \abs{ \widehat{\beta}(s) - \beta(s) } 
    \leq& \Cue \sqrt{\frac{\tmix + 2s}{n-1} \calJ_{\infty}^{(s)}}, \\
\end{split}
\end{equation*}
where $\tmix$ is the mixing time of $P$ and $\calJ_\infty^{(s)}$ is defined in \eqref{definition:J-infty}.
\end{theorem}

\begin{theorem}[$(\eps,\delta)$-PAC Bound]
\label{theorem:estimation-beta-coefficients-uniform-ergodicity-skip-wise-pac}
Let $\eps, \delta \in (0,1), s, n \in \bbN$ with $n > 2s + 1$, and let $\widehat{\beta}(s) \colon \calX^n \to [0,1]$ be the estimator  defined in \eqref{eq:estimator-beta-s}.
There exists a universal constant $\Cue \leq 3\sqrt{2}$ such that
for any $\pi$-stationary uniformly ergodic Markov chain $X_1, X_2, \dots, X_n$ with transition operator $P \in \calW^\star(\calX)$ such that
$\calJ_{\infty}^{(s)}<\infty$, for
    \begin{equation*}
        n \geq 1 + \frac{4(\tmix + 2s)}{\eps^2} \max \set{ \Cue^2  \calJ_\infty^{(s)}, 1152 \log \frac{2}{\delta}},
    \end{equation*}
    with probability at least $1 - \delta$, it holds that
    \begin{equation*}
        \abs{\widehat{\beta}(s) - \beta(s)} \leq \eps.
    \end{equation*}
\end{theorem}

\begin{remark}
\label{remark:techniques-are-different}
    Specializing Theorem~\ref{theorem:estimation-beta-coefficients-skip-wise-without-tmix} to obtain a PAC bound in the uniformly ergodic setting would instead yield an inferior PAC upper bound in
    $$n \geq \frac{C}{\eps^2} \max\set{ \left(s + p \tmix  \right)\calJ_{p}^{(s)} ,  (s +  \tmix ) \log \frac{1}{\delta} }, C > 0.$$
\end{remark}

\subsubsection{Comparison with prior results}
\label{section:estimation-comparison-prior-results}

In \citet[Theorem~4]{mcdonald2015estimating}, the authors establish the nearly parametric convergence rate for estimating $\beta$-mixing coefficients of continuously valued Markov processes
\begin{equation*}
    \bbE{\abs{\widehat{\beta}(s) - \beta(s)}} \leq \bigO \left( \sqrt{\frac{W_0(n)}{n}} \right),
\end{equation*}
where $W_0$ is the Lambert function \citep{corless1996lambert}.
They also posit that their findings may be considered conservative when applied to discrete-valued processes. Theorem~\ref{theorem:estimation-beta-coefficients-skip-wise-without-tmix} validates their remark by recovering the superior parametric rate $\bigO(\sqrt{1/n})$ under convergence of $\calB_p^{(s)}$ and $\calJ_p^{(s)}$.
In the general ergodic setting, our method shares similarities with 
\citet{khaleghi2023inferring}, as both employ Rio's covariance inequality  \citet[Corollary~1.1]{rio1999theorie}.
However, compared to both \citet{khaleghi2023inferring} and \citet[Theorem~4]{mcdonald2015estimating},
 we distinguish ourselves by focusing on the discrete space setting and deriving explicit rates in terms of fine-grained properties of the transition operator.
We also note that \citet{grunewalder2024estimating} address the estimation of $\beta$-mixing coefficients in the geometrically ergodic setting, while we consider a broader class of $\beta$-mixing processes, including sub-exponentially and algebraically $\beta$-mixing ones. 
Finally, under uniform ergodicity, instead of \citeauthor{rio1999theorie}'s covariance inequality, we rely on \citet{paulin2015concentration} which yields a stronger bound on the MAD and sharper PAC bounds (refer to Remark~\ref{remark:techniques-are-different}).

\subsection{Estimation of the average-mixing time}
\label{section:estimation-average-mixing-time}
We now show how the sequence of estimators for the $\beta$-mixing coefficients, discussed in the previous section, translates into an estimator for the average-mixing time.
Recall that the average distance to stationarity $\beta$ does not generally enjoy sub-multiplicativity. In turn it $\atmix(\xi)$ cannot generally be bounded by $\atmix(1/4)$ and a polynomial of $\log (1/\xi)$. Consequently, we formalize desired guarantees on our estimator  differently from \citet{hsu2019mixing, wolfer2024improved} for the relaxation time, where the error is defined multiplicatively with respect to the estimator $| \widehat{\trel} / \trel - 1 | \leq \eps$.
When estimating the average-mixing time, we treat the proximity $\xi$ to stationarity as a free parameter, and construct an estimator that targets a band $\xi(1 \pm \eps)$.
Specifically, we define for $\xi \in (0,1)$,
\begin{equation}
\label{eq:estimator-average-tmix}
    \widehat{\tmix}^{\sharp}(\xi) \eqdef \argmin_{s \in \bbN} \set{ \widehat{\beta}(s) \leq \xi },
\end{equation}
where $\widehat{\beta}(s)$ is the estimator for $\beta(s)$ defined in \eqref{eq:estimator-beta-s}, and our next objective is to obtain with high probability an upper bound on the trajectory length such that
\begin{equation*}
\begin{split}
\widehat{\tmix}^{\sharp}(\xi) \in \left[ \tmix^{\sharp}(\xi(1 + \eps)), \tmix^{\sharp}(\xi(1 - \eps)) \right].
\end{split}
\end{equation*}
As illustrated on Figure~\ref{fig:average-mixing-time-estimation-window}, the interval above becomes thinner as $\eps$ decays to zero.
Our next lemma explains how to convert a sufficient trajectory length for high-probability control of the error of the sequence of estimators for the $\beta$-mixing coefficients into a sufficient trajectory length in order to estimate $\atmix(\Theta(\xi))$.

\begin{figure}
    \centering

\tikzset{every picture/.style={line width=0.75pt}} %

\begin{tikzpicture}[x=0.68pt,y=0.68pt,yscale=-1,xscale=1]

\draw  (82,221) -- (566,221)(90,9) -- (90,233) (559,216) -- (566,221) -- (559,226) (85,16) -- (90,9) -- (95,16)  ;
\draw [color=wred  ,draw opacity=1 ][line width=1.5]    (90.5,141.5) -- (562.5,141.5) ;
\draw [color=wred  ,draw opacity=1 ][line width=0.75]    (89.5,162) -- (562.5,162) ;
\draw [color=wred  ,draw opacity=1 ][line width=0.75]    (90,122) -- (518,122) -- (562.5,122) ;
\draw  [draw opacity=0][fill=wred  ,fill opacity=0.2 ] (90,122) -- (562.5,122) -- (562.5,162) -- (90,162) -- cycle ;
\draw    (90,221.27) -- (561,221) (110,217.26) -- (110,225.26)(130,217.25) -- (130,225.25)(150,217.23) -- (150,225.23)(170,217.22) -- (170,225.22)(190,217.21) -- (190,225.21)(210,217.2) -- (210,225.2)(230,217.19) -- (230,225.19)(250,217.18) -- (250,225.18)(270,217.17) -- (270,225.17)(290,217.15) -- (290,225.15)(310,217.14) -- (310,225.14)(330,217.13) -- (330,225.13)(350,217.12) -- (350,225.12)(370,217.11) -- (370,225.11)(390,217.1) -- (390,225.1)(410,217.09) -- (410,225.09)(430,217.07) -- (430,225.07)(450,217.06) -- (450,225.06)(470,217.05) -- (470,225.05)(490,217.04) -- (490,225.04)(510,217.03) -- (510,225.03)(530,217.02) -- (530,225.02)(550,217.01) -- (550,225.01) ;
\draw   (108,18.5) .. controls (108,17.12) and (109.12,16) .. (110.5,16) .. controls (111.88,16) and (113,17.12) .. (113,18.5) .. controls (113,19.88) and (111.88,21) .. (110.5,21) .. controls (109.12,21) and (108,19.88) .. (108,18.5) -- cycle ;
\draw   (127,42.5) .. controls (127,41.12) and (128.12,40) .. (129.5,40) .. controls (130.88,40) and (132,41.12) .. (132,42.5) .. controls (132,43.88) and (130.88,45) .. (129.5,45) .. controls (128.12,45) and (127,43.88) .. (127,42.5) -- cycle ;
\draw   (148,68.5) .. controls (148,67.12) and (149.12,66) .. (150.5,66) .. controls (151.88,66) and (153,67.12) .. (153,68.5) .. controls (153,69.88) and (151.88,71) .. (150.5,71) .. controls (149.12,71) and (148,69.88) .. (148,68.5) -- cycle ;
\draw   (168,89.5) .. controls (168,88.12) and (169.12,87) .. (170.5,87) .. controls (171.88,87) and (173,88.12) .. (173,89.5) .. controls (173,90.88) and (171.88,92) .. (170.5,92) .. controls (169.12,92) and (168,90.88) .. (168,89.5) -- cycle ;
\draw   (188,104.5) .. controls (188,103.12) and (189.12,102) .. (190.5,102) .. controls (191.88,102) and (193,103.12) .. (193,104.5) .. controls (193,105.88) and (191.88,107) .. (190.5,107) .. controls (189.12,107) and (188,105.88) .. (188,104.5) -- cycle ;
\draw   (208,116.5) .. controls (208,115.12) and (209.12,114) .. (210.5,114) .. controls (211.88,114) and (213,115.12) .. (213,116.5) .. controls (213,117.88) and (211.88,119) .. (210.5,119) .. controls (209.12,119) and (208,117.88) .. (208,116.5) -- cycle ;
\draw   (227,128.5) .. controls (227,127.12) and (228.12,126) .. (229.5,126) .. controls (230.88,126) and (232,127.12) .. (232,128.5) .. controls (232,129.88) and (230.88,131) .. (229.5,131) .. controls (228.12,131) and (227,129.88) .. (227,128.5) -- cycle ;
\draw   (247,137.5) .. controls (247,136.12) and (248.12,135) .. (249.5,135) .. controls (250.88,135) and (252,136.12) .. (252,137.5) .. controls (252,138.88) and (250.88,140) .. (249.5,140) .. controls (248.12,140) and (247,138.88) .. (247,137.5) -- cycle ;
\draw   (267,146.5) .. controls (267,145.12) and (268.12,144) .. (269.5,144) .. controls (270.88,144) and (272,145.12) .. (272,146.5) .. controls (272,147.88) and (270.88,149) .. (269.5,149) .. controls (268.12,149) and (267,147.88) .. (267,146.5) -- cycle ;
\draw   (288,153.5) .. controls (288,152.12) and (289.12,151) .. (290.5,151) .. controls (291.88,151) and (293,152.12) .. (293,153.5) .. controls (293,154.88) and (291.88,156) .. (290.5,156) .. controls (289.12,156) and (288,154.88) .. (288,153.5) -- cycle ;
\draw   (308,158.5) .. controls (308,157.12) and (309.12,156) .. (310.5,156) .. controls (311.88,156) and (313,157.12) .. (313,158.5) .. controls (313,159.88) and (311.88,161) .. (310.5,161) .. controls (309.12,161) and (308,159.88) .. (308,158.5) -- cycle ;
\draw   (328,165.5) .. controls (328,164.12) and (329.12,163) .. (330.5,163) .. controls (331.88,163) and (333,164.12) .. (333,165.5) .. controls (333,166.88) and (331.88,168) .. (330.5,168) .. controls (329.12,168) and (328,166.88) .. (328,165.5) -- cycle ;
\draw   (348,172.5) .. controls (348,171.12) and (349.12,170) .. (350.5,170) .. controls (351.88,170) and (353,171.12) .. (353,172.5) .. controls (353,173.88) and (351.88,175) .. (350.5,175) .. controls (349.12,175) and (348,173.88) .. (348,172.5) -- cycle ;
\draw   (368,178.5) .. controls (368,177.12) and (369.12,176) .. (370.5,176) .. controls (371.88,176) and (373,177.12) .. (373,178.5) .. controls (373,179.88) and (371.88,181) .. (370.5,181) .. controls (369.12,181) and (368,179.88) .. (368,178.5) -- cycle ;
\draw   (388,182.5) .. controls (388,181.12) and (389.12,180) .. (390.5,180) .. controls (391.88,180) and (393,181.12) .. (393,182.5) .. controls (393,183.88) and (391.88,185) .. (390.5,185) .. controls (389.12,185) and (388,183.88) .. (388,182.5) -- cycle ;
\draw   (408,185.5) .. controls (408,184.12) and (409.12,183) .. (410.5,183) .. controls (411.88,183) and (413,184.12) .. (413,185.5) .. controls (413,186.88) and (411.88,188) .. (410.5,188) .. controls (409.12,188) and (408,186.88) .. (408,185.5) -- cycle ;
\draw   (428,188.5) .. controls (428,187.12) and (429.12,186) .. (430.5,186) .. controls (431.88,186) and (433,187.12) .. (433,188.5) .. controls (433,189.88) and (431.88,191) .. (430.5,191) .. controls (429.12,191) and (428,189.88) .. (428,188.5) -- cycle ;
\draw   (448,192.5) .. controls (448,191.12) and (449.12,190) .. (450.5,190) .. controls (451.88,190) and (453,191.12) .. (453,192.5) .. controls (453,193.88) and (451.88,195) .. (450.5,195) .. controls (449.12,195) and (448,193.88) .. (448,192.5) -- cycle ;
\draw [color={rgb, 255:red, 155; green, 155; blue, 155 }  ,draw opacity=1 ] [dash pattern={on 0.84pt off 2.51pt}]  (229.5,128.5) -- (230,221) ;
\draw [color={rgb, 255:red, 155; green, 155; blue, 155 }  ,draw opacity=1 ] [dash pattern={on 0.84pt off 2.51pt}]  (269.5,146.5) -- (270,221) ;
\draw [color={rgb, 255:red, 155; green, 155; blue, 155 }  ,draw opacity=1 ] [dash pattern={on 0.84pt off 2.51pt}]  (330.5,165.5) -- (330.5,221.13) ;
\draw   (468,195.5) .. controls (468,194.12) and (469.12,193) .. (470.5,193) .. controls (471.88,193) and (473,194.12) .. (473,195.5) .. controls (473,196.88) and (471.88,198) .. (470.5,198) .. controls (469.12,198) and (468,196.88) .. (468,195.5) -- cycle ;
\draw   (488,198.5) .. controls (488,197.12) and (489.12,196) .. (490.5,196) .. controls (491.88,196) and (493,197.12) .. (493,198.5) .. controls (493,199.88) and (491.88,201) .. (490.5,201) .. controls (489.12,201) and (488,199.88) .. (488,198.5) -- cycle ;
\draw   (508,200.5) .. controls (508,199.12) and (509.12,198) .. (510.5,198) .. controls (511.88,198) and (513,199.12) .. (513,200.5) .. controls (513,201.88) and (511.88,203) .. (510.5,203) .. controls (509.12,203) and (508,201.88) .. (508,200.5) -- cycle ;
\draw   (528,202.5) .. controls (528,201.12) and (529.12,200) .. (530.5,200) .. controls (531.88,200) and (533,201.12) .. (533,202.5) .. controls (533,203.88) and (531.88,205) .. (530.5,205) .. controls (529.12,205) and (528,203.88) .. (528,202.5) -- cycle ;
\draw   (548,203.5) .. controls (548,202.12) and (549.12,201) .. (550.5,201) .. controls (551.88,201) and (553,202.12) .. (553,203.5) .. controls (553,204.88) and (551.88,206) .. (550.5,206) .. controls (549.12,206) and (548,204.88) .. (548,203.5) -- cycle ;

\draw (73,133) node [anchor=north west][inner sep=0.75pt]   [align=left] {$\xi$};
\draw (30,113) node [anchor=north west][inner sep=0.75pt]   [align=left] {$(1+\eps)\xi$};
\draw (30,153) node [anchor=north west][inner sep=0.75pt]   [align=left] {$(1-\eps)\xi$};
\draw (56,17) node [anchor=north west][inner sep=0.75pt]   [align=left] {$\beta(s)$};
\draw (547,227) node [anchor=north west][inner sep=0.75pt]   [align=left] {$s$};
\draw (270,225) node [anchor=north west][inner sep=0.75pt] [rotate=-30,xslant=-0.07]  [align=left] {$\atmix(\xi)$};
\draw (330,225) node [anchor=north west][inner sep=0.75pt] [rotate=-30,xslant=-0.07]  [align=left] {$\atmix(\xi(1 - \eps))$};
\draw (230,225) node [anchor=north west][inner sep=0.75pt] [rotate=-30,xslant=-0.07]  [align=left] {$\atmix(\xi(1 + \eps))$};

\end{tikzpicture}
    
    \caption{Average-mixing time $\xi(1 \pm \eps)$ band.}
    \label{fig:average-mixing-time-estimation-window}
\end{figure}

\begin{lemma}
    \label{lemma:convert-beta-pac-to-atmix-bound}
    Let $\delta, \xi, \eps \in (0,1)$ and $n \in \bbN$.
Suppose that for any $s\in \bbN$, there exists an estimator $\widehat{\beta}(s)$ for $\beta(s)$ such that for any $\pi$-stationary ergodic Markov chain $X_1, \dots, X_n$, when $n \geq n_0(P, s, \eps, \delta)$ it holds with probability at least $1 - \delta$ that
\begin{equation*}
    \abs{\widehat{\beta}(s) - \beta(s)} < \eps.
\end{equation*} 
Let $\widehat{\tmix}^{\sharp}(\xi) \colon \calX^n \to \bbN$ be the estimator for $\atmix(\xi)$ defined in 
\eqref{eq:estimator-average-tmix}.
Then for
\begin{equation*}
    n \geq \max_{1\leq s \le \atmix(\xi(1 - \eps))} n_0\left(P, s, \xi \eps, \frac{\delta}{2 \atmix(\xi)}\right),
\end{equation*}
with probability at least $1 - \delta$, it holds that 
\begin{equation*}
\begin{split}
\widehat{\tmix}^{\sharp}(\xi) \in \left[ \tmix^{\sharp}(\xi(1 + \eps)), \tmix^{\sharp}(\xi(1 - \eps)) \right].
\end{split}
\end{equation*}
\end{lemma}

\begin{theorem}
\label{theorem:estimation-average-mixing-time}
Let $\delta, \xi, \eps \in (0,1)$,
    $n \in \bbN$, and let $\widehat{\tmix}^{\sharp}(\xi) \colon \calX^n \to \bbN$ be the estimator for $\atmix(\xi)$ defined in 
\eqref{eq:estimator-average-tmix}. There exist universal constants $\Ce, \Cue \in \bbR_+$ such that the following holds.
Let $X_1, X_2, \dots, X_n$ be a $\pi$-stationary Markov chain with transition operator $P \in \calW(\calX)$.
For any $p \in [1, \infty)$, we write
\begin{equation*}
    \calJ_{p, \xi} \eqdef \sup \set{ \calJ_p^{(s)} \colon s \in \bbN, s \leq \atmix(\xi) }.
\end{equation*}
\begin{enumerate}
\item When $P$ is ergodic there exists $p \geq 1$ such that $\set{ \calB_p \calJ_{p, \xi(1 - \eps)}  } < \infty,$ and
\begin{equation*}
    n
    \geq
    1
    +
    \Ce^2
    \left[
        \atmix\left(\xi(1-\eps)\right)+
        \atmix\left(
            \eta\left(
                \eps\xi,
                \frac{\delta}{2\atmix(\xi)}
            \right)
        \right)
    \right]
    \frac{1}{\xi^4\eps^4}
    \left(
        1+ C_{\eps, \xi}
    \right)
    \log\frac{4\atmix(\xi)}{\delta},
\end{equation*}
where
\begin{equation*}
    \eta(\eps,\delta) = \frac{ \eps^2 \delta}{\Ce \log(8/\delta)}, \qquad C_{\eps, \xi} = \inf_{p\geq1}
        \set{
            \calB_p\calJ_{p,\xi(1-\eps)}
        }.
\end{equation*}
\item or when $P$ is uniformly ergodic, $\calJ_{\infty, \xi(1 - \eps)} < \infty$, and
    \begin{equation*}
        n \geq 1 + 4 \tmix\frac{1 + 2\ceil*{\log_2\frac{1}{\xi(1 - \eps)}}}{\xi^2 \eps^2} \max \set{ \Cue^2 \calJ_{\infty, \xi(1 - \eps)}, 1152 \log \frac{4 \tmix \ceil*{\log_2 1/\xi}}{\delta}},
    \end{equation*}
\end{enumerate}
then with probability at least $1 - \delta$, it holds that
\begin{equation*}
\begin{split}
\widehat{\tmix}^{\sharp}(\xi) \in \left[ \tmix^{\sharp}(\xi(1 + \eps)), \tmix^{\sharp}(\xi(1 - \eps)) \right].
\end{split}
\end{equation*}

\end{theorem}

\begin{proof}
    The theorem follows from plugging 
Theorem~\ref{theorem:estimation-beta-coefficients-general-ergodicity-skip-wise} and Theorem~\ref{theorem:estimation-beta-coefficients-uniform-ergodicity-skip-wise-pac} into
Lemma~\ref{lemma:convert-beta-pac-to-atmix-bound}.
The theorem in the uniformly ergodic case follows from $$\atmix(\xi) \leq \tmix(\xi) \leq \tmix \ceil*{\log_2 \frac{1}{\xi}}.
$$
\end{proof}

\begin{corollary}[Confidence interval for $\atmix(\xi)$]
As a consequence of Theorem~\ref{theorem:estimation-average-mixing-time}, we can construct a confidence interval for $\atmix(\xi)$. Denoting $n_0(\xi, \delta)$ the sample complexity in Theorem~\ref{theorem:estimation-average-mixing-time}, when $$n \geq \max \set{ n_0\left( \frac{\xi}{1 - \eps}, \delta/2 \right), n_0\left( \frac{\xi}{1 + \eps}, \delta/2 \right) },$$
it holds with probability $1 - \delta$ that
\begin{equation*}
  \atmix(\xi) \in \left[ \widehat{\tmix}^{\sharp}\left(\frac{\xi}{1 - \eps}\right), \widehat{\tmix}^{\sharp}\left(\frac{\xi}{1 + \eps}\right) \right].
\end{equation*}
\end{corollary}

In the non uniformly ergodic setting, our method introduces a trade-off parameter $p$ to obtain an upper bound of order
\begin{equation*}
      \widetilde \bigO \left( \frac{\atmix}{\xi^4} \inf_{p \geq 1} \set{1 + \calB_p \calJ_{p, \xi} } \right).
\end{equation*}
This enables the use of versatile bounds on the entropic term, which we explore in Section~\ref{section:implications-various-state-space-scales}.
Currently, there are no related work to which we can compare the general results obtained in Section~\ref{section:estimation-average-mixing-time}. We defer a comparison to the state of the art in the finite space setting to Section~\ref{section:comparison-relaxation-time-mixing-time-estimation-complexity}.

\section{Implications at various scales of state space}
\label{section:implications-various-state-space-scales}
Recall (refer to Theorem~\ref{theorem:estimation-beta-coefficients-skip-wise-without-tmix}, Theorem~\ref{theorem:estimation-beta-coefficients-general-ergodicity-skip-wise}) that we require a bound on the entropic term 
\begin{equation*}
    \calJ_p^{(s)} = \nrm{Q^{(s)}}^{1 - 1/p}_{(1 - 1/p)/2},
\end{equation*}
where for $s \in \bbN$ and for $x,x' \in \calX$, we wrote $Q^{(s)}(x,x') = \pi(x) P^s(x,x')$, and for $q \in \bbR_+$,
\begin{equation*}
    \nrm{Q^{(s)}}^{q}_{q} \eqdef \sum_{(x,x')\in \calX^2} Q^{(s)}(x,x')^q,
\end{equation*}
in order to recover an estimation rate for the individual $\beta$-mixing coefficients,
whilst upper bounds on the sample complexity for estimating the average-mixing time (Theorem~\ref{theorem:estimation-average-mixing-time}) require
 a control which holds uniformly over all skipping rates up to the order of the average-mixing time,
\begin{equation}
\label{eq:entropic-term}
    \calJ_{p, \xi} = \sup \set{ \calJ^{(s)}_p \colon s \in \bbN, s \leq \atmix (\xi)}.
\end{equation}
Bounds generally depend on the tuning parameter $p$, which can then be appropriately selected---for instance in Theorem~\ref{theorem:estimation-average-mixing-time}---in order to optimize the upper bound on the sample complexity.
Unfortunately, $\calJ_{p, \xi}$ is generally not easily amenable to analysis; hence, our focus lies in establishing bounds for its evaluation.
It is not necessarily the case that $\calJ_p^{(s)}$ is monotonous in $s$.
As illustrated in Figure~\ref{fig:entropy-transient-regime}, $\calJ_p^{(s)}$ can exhibit a transient regime.
In this case, under ergodicity of $P$, the pair-wise limit
\begin{equation*}
    \lim_{s \to \infty} Q^{(s)}(x,x') = \pi(x)\pi(x'), \qquad \forall (x,x') \in \calX^2,
\end{equation*}
hints at a heuristic lower bound in $\nrm{\pi}_{(1 - 1/p)/2}^{2(1 - 1/p)}$ on $\calJ_{p, \xi}$.
Using the stationary regime as a baseline, it seems natural to examine whether the entropic term undergoes relative explosive growth in the transient regime.
For instance, for some multiplicative factor $K \geq 1$, an upper bound
\begin{equation}
\label{eq:optimistic-bound-entropic-term}
    \calJ_p^{(s)} \leq K \nrm{\pi}^{2(1 - 1/p)}_{(1 - 1/p)/2},
\end{equation}
with a controlled $K$
readily streamlines the sample complexities appearing in Section~\ref{section:estimation} in terms of properties of the stationary distribution.
Let us illustrate the approach by assuming the existence of a uniformly dominating distribution over all conditional distributions. Formally, let us suppose that there exists a universal constant $C \in \bbR_+$ and a probability measure $\nu \in \calP(\calX)$, such that for all $x' \in \calX$,
$$ \sup_{x \in \calX} P(x,x')\leq C \nu(x') .$$
It follows that
\begin{equation*}
\begin{split}
 Q^{(s)}(x,x') &\leq \pi(x) \sum_{x'' \in \calX} C \nu(x')  P^{s-1}(x,x'') =  C\pi(x) \nu(x'),
\end{split}
\end{equation*}  
and thus
\begin{equation*}
 \sup_{s \in \bbN} \calJ_{p}^{(s)} \leq C^{1 - 1/p} \nrm{\nu}_{(1 - 1/p)/2}^{1 - 1/p} \nrm{\pi}_{(1 - 1/p)/2}^{1 - 1/p}.
\end{equation*}
In the rest of the section, we proceed to discuss bounds on the entropic quantities $\calJ_p^{(s)}$ and $\calJ_{p,\xi}$ in terms of ergodic properties of the chain, structural properties of the connection graph of the Markov chain, or the alphabet size.
We conclude with a discussion on the binary state space, where the dominant complexity stems from the mixing bottlenecks, revealing that our approach holds implications at every scale.

\begin{remark}[Convention]
    Throughout this section, whenever a bound or finite sample-complexity is contingent on a quantity (e.g. $\calJ_p^{(s)}, \dots$) or a semi-norm (e.g. $\nrm{\pi}_{1/2}$ , \dots) being finite, the statement is understood to be conditioned on finiteness of that quantity.
\end{remark}

\begin{figure}
    \centering
    
\tikzset{every picture/.style={line width=0.75pt}} %

\begin{tikzpicture}[x=0.50pt,y=0.50pt,yscale=-1,xscale=1]

\draw  (50,240.67) -- (596,240.67)(59.82,22.67) -- (59.82,250) (589,235.67) -- (596,240.67) -- (589,245.67) (54.82,29.67) -- (59.82,22.67) -- (64.82,29.67) (79.82,235.67) -- (79.82,245.67)(99.82,235.67) -- (99.82,245.67)(119.82,235.67) -- (119.82,245.67)(139.82,235.67) -- (139.82,245.67)(159.82,235.67) -- (159.82,245.67)(179.82,235.67) -- (179.82,245.67)(199.82,235.67) -- (199.82,245.67)(219.82,235.67) -- (219.82,245.67)(239.82,235.67) -- (239.82,245.67)(259.82,235.67) -- (259.82,245.67)(279.82,235.67) -- (279.82,245.67)(299.82,235.67) -- (299.82,245.67)(319.82,235.67) -- (319.82,245.67)(339.82,235.67) -- (339.82,245.67)(359.82,235.67) -- (359.82,245.67)(379.82,235.67) -- (379.82,245.67)(399.82,235.67) -- (399.82,245.67)(419.82,235.67) -- (419.82,245.67)(439.82,235.67) -- (439.82,245.67)(459.82,235.67) -- (459.82,245.67)(479.82,235.67) -- (479.82,245.67)(499.82,235.67) -- (499.82,245.67)(519.82,235.67) -- (519.82,245.67)(539.82,235.67) -- (539.82,245.67)(559.82,235.67) -- (559.82,245.67)(579.82,235.67) -- (579.82,245.67)(54.82,220.67) -- (64.82,220.67)(54.82,200.67) -- (64.82,200.67)(54.82,180.67) -- (64.82,180.67)(54.82,160.67) -- (64.82,160.67)(54.82,140.67) -- (64.82,140.67)(54.82,120.67) -- (64.82,120.67)(54.82,100.67) -- (64.82,100.67)(54.82,80.67) -- (64.82,80.67)(54.82,60.67) -- (64.82,60.67)(54.82,40.67) -- (64.82,40.67) ;
\draw   ;
\draw    (60,240.67) .. controls (100,210.67) and (110,80.67) .. (160,80.67) .. controls (210,80.67) and (236,161.67) .. (590,160.67) ;
\draw [color=wred  ,draw opacity=1 ] [dash pattern={on 0.84pt off 2.51pt}]  (59,81) -- (161,80.67) ;
\draw [color=wred  ,draw opacity=1 ] [dash pattern={on 0.84pt off 2.51pt}]  (160,241) -- (160,80.67) ;
\draw [color={rgb, 255:red, 155; green, 155; blue, 155 }  ,draw opacity=1 ] [dash pattern={on 0.84pt off 2.51pt}]  (590,160.67) -- (59,160.33) ;
\draw [color={rgb, 255:red, 155; green, 155; blue, 155 }  ,draw opacity=1 ] [dash pattern={on 0.84pt off 2.51pt}]  (240,240.33) -- (240,113.33) ;

\draw (10,65) node [anchor=north west][inner sep=0.75pt] [color=wred  ,draw opacity=1 ]  [align=left] {$\calJ_{p, \xi}$};
\draw (10,150) node [anchor=north west][inner sep=0.75pt]   [align=left] {$\calJ_{\infty}$};
\draw (563,250) node [anchor=north west][inner sep=0.75pt]   [align=left] {$s$};
\draw (225,245) node [anchor=north west][inner sep=0.75pt]   [align=left] {$\atmix(\xi)$};
\draw (10,15) node [anchor=north west][inner sep=0.75pt]   [align=left] {$\calJ_p^{(s)}$};

\end{tikzpicture}
    \caption{Transient regime for the entropic term $\calJ_p^{(s)}$.}
    \label{fig:entropy-transient-regime}
\end{figure}

\subsection{Ergodic properties}
\label{subsec:implications-ergodic}
In this subsection, we derive a control of $\calJ_p^{(s)} $ based on a decomposition and natural ergodic properties of the chain that we introduce.
Let us first  decompose $\sqrt{\calJ_p^{(s)}}$ in terms of a stationary quantity and some measure of distance to stationarity. Recall that for 
$a,b > 0, q \in (0,1)$, it holds that $(a + b)^q \leq  a^q + b^q$,
thus
\begin{align*}
 \sqrt{\calJ_p^{(s)}} &\leq 2 \beta_{(1 - 1/p)/2}(s) + \nrm{\pi}_{(1 - 1/p)/2}^{1 - 1/p},
\end{align*}
where for $q \in (0,1)$, $\beta_q(s) \geq \beta(s)$ is a generalization of the $\beta$-mixing coefficient defined as
\begin{equation*}
\begin{split}
\beta_q(s) &\eqdef \frac{1}{2} \sum_{x, x' \in \calX}  \abs{\pi(x)P^s(x,x') - \pi(x)\pi(x')}^{q}.\\
\end{split}
\end{equation*}
Moving forward, for a function $V \colon \calX \to [1,\infty)$ we introduce the notion of the $V_{q}$-norm of a signed measure as follows,
\begin{equation*}
    \nrm{\mu}_{q,V}^q \eqdef \sum_{x \in \calX} \abs{\mu(x)}^q V(x),
\end{equation*}
and we say that $P$ is $V_q$-geometrically ergodic when there exists $\rho_{q, V} \in (0,1)$ such that for all $s \in \bbN$,
\begin{equation}
\label{definition:v-geometric-ergodicity}
    \sup_{x \in \calX } \frac{\nrm{e_x P^s - \pi}^q_{q,V}}{V(x)} \leq \rho_{q, V}^s.
\end{equation}
Under this assumption the coefficients $\beta(s)$ and $\beta_{q}(s)$ can be bounded as follows,
\begin{equation*}
\begin{split}
\beta(s) \leq \beta_q(s) &= \frac{1}{2} \sum_{x,x' \in \calX}  \abs{\pi(x)P^s(x,x') - \pi(x)\pi(x')}^q \\
&\leq \frac{1}{2} \sum_{x \in \calX} \pi(x)^q \sum_{x' \in \calX} \abs{P^s(x,x') - \pi(x')}^q V(x') \\
&\leq \frac{1}{2} \sum_{x \in \calX} \pi(x)^q V(x) \rho_{q, V}^s = \frac{1}{2} \rho_{q, V}^s \nrm{\pi}_{q, V}^q, \\
\end{split}
\end{equation*}
thus the average-mixing time satisfies,
\begin{equation*}
    \atmix(\xi) \leq 1 \lor \Bigg\lceil \cfrac{\log  \nrm{\pi}_{q,V}^q + \log \frac{1}{2\xi} }{\log \rho_{q,V}^{-1}} \Bigg\rceil,
\end{equation*}
As a result, writing for simplicity, $\Pi_{r, V} \eqdef \nrm{\pi}_{r, V}^{r}$ for $r \in (0,1)$, for any $q \in (0, (1-1/p)/2]$ we can bound the entropic term as follows,
\begin{equation*}
\begin{split}
    \calJ_p^{(s)} \leq \left(  \rho_{q, V}^s \Pi_{q, V} + \Pi_{(1 - 1/p)/2, 1}^2 \right)^2,
\end{split}
\end{equation*}
leading to the following theorem.

\begin{theorem}[Under $V_q$-geometrical ergodicity]
\label{theorem:entropy-term-under-vq-geometric-ergodicity}
    Let $p \geq 1$.
    If there exists $q \in (0, (1 - 1/p)/2]$ such $P$ is $V_{q}$-geometrically ergodic---refer to \eqref{definition:v-geometric-ergodicity}---for any $s \in \bbN$, it holds that
    \begin{equation*}
\begin{split}
    \calJ_p^{(s)} \leq \left(  \rho_{q, V}^s \Pi_{q, V} + \Pi_{(1 - 1/p)/2, 1}^2 \right)^2.
\end{split}
\end{equation*}
\end{theorem}
Observe from the above that, as soon as
\begin{equation*}
\begin{split}
    s  &\geq \cfrac{\log \Pi_{q, V} - 2 \log \Pi_{(1 - 1/p)/2, 1}}{\log \rho_{q, V}^{-1}}, \\
\end{split}
\end{equation*}
we obtain a control of the entropic term in terms of a purely stationary quantity 
\begin{equation*}
\begin{split}
    \calJ_p^{(s)} \leq 4 \Pi_{(1 - 1/p)/2, 1}^4,
\end{split}
\end{equation*}
giving us insight into the exit time of the transient regime illustrated in Figure~\ref{fig:entropy-transient-regime}.
Finally, a similar proof allows us to analyze the entropic term under the following condition, which is much weaker than $V$-geometrical ergodicity.
\begin{theorem}
    \label{theorem:entropy-term-under-pointwise-geometric-ergodicity}
    Suppose that there exist $\rho \in (0,1)$ and $V \colon \calX \to [1, \infty)$ such that for any $s \in \bbN$,
     \begin{equation*}
         \sup_{x,x'} \abs{P^s(x,x') - \pi(x')} \frac{V(x')}{V(x)} \leq \rho^s.
     \end{equation*}
     Then for all $p \geq 1$ and for all
     $s \in \bbN$, it holds that
    \begin{equation*}
\begin{split}
    \calJ_p^{(s)} \leq \left(  \rho^{s(1 - 1/p)/2} \nrm{ \pi \odot V }_{(1 - 1/p)/2}^{(1 - 1/p)/2}\nrm{1/V}_{(1 - 1/p)/2}^{(1 - 1/p)/2} + \Pi_{(1 - 1/p)/2, 1}^2 \right)^2.
\end{split}
\end{equation*}
\end{theorem}

\begin{remark}
    Obvious summability conditions on $\nrm{1/V}_{(1 - 1/p)/2}$ and $\nrm{\pi \odot V}_{(1 - 1/p)/2}$ are necessary for the bound in the above theorem to be non-vacuous. Additionally, under the condition of Theorem~\ref{theorem:entropy-term-under-pointwise-geometric-ergodicity} a bound on the $\beta$-mixing coefficient can be immediately obtained as follows,
    \begin{equation*}
        \beta(s) \leq \frac{\rho^s}{2} \nrm{\pi \odot V}_1 \nrm{ 1/V }_1.
    \end{equation*}
\end{remark}

\subsection{Infinite graphs with controlled growth}
\label{section:controlled-growth-graphs}
Natural structural assumptions on 
the underlying connection graph of a Markov chain also enable us to control the entropic term.
Recall that an ergodic transition operator $P \in \calW(\calX)$ induces a strongly connected digraph $(\calX, \calE)$, where
$$\calE \eqdef \set{ (x,x') \in \calX^2 \colon P(x,x') > 0}$$
is the set of directed edges.
Adapting the definition of growth of a graph \citep{trofimov1985graphs, campbell2022graphs} to the directed setting, we define the growth of $(\calX, \calE)$ as the function $g \colon \bbN \to \bbN \cup \set{\infty}$, where for $s \in \bbN$,
$$g(s) \eqdef \sup_{x \in \calX} \abs{ \set{ x' \in \calX \colon \rho(x,x') \leq s} },$$
where $\rho \colon \calX \times \calX \to \bbN \cup \set{0, \infty}$ denotes the directed graph distance.

\begin{condition}[Connection graph with polynomial growth]
\label{condition:connection-graph-with-polynomial-growth}
    A graph is said to have at most polynomial growth when there exists $G \in \bbR_+$ and $q \geq 1$ such that for any $s \in \bbN$, it holds that $g(s) \leq G s^q$. Then $q$ and $G$ are respectively called the degree and factor of growth.
\end{condition}

\begin{example}[Birth and death chains]
    It is immediate from its definition that the connection graph associated with the Chebyshev type random walk developed in Example~\ref{example:birth-and-death} has bounded growth with $G = 3$ and $q = 1$.
\end{example}

\begin{lemma}
    \label{lemma:from-polynomial-growth}
Let $p, q \in [1, \infty), \xi \in (0, 1)$ and let $P \in \calW(\calX)$ be a $\pi$-stationary transition operator. If the connection graph of $P$ has polynomial growth of degree $q$ with factor $G$, it holds that
\begin{equation*}
    \calJ_{p, \xi} \leq 
    G^{1 + 1/p} \atmix(\xi)^{q(1 + 1/p)} \nrm{\pi}^{1 - 1/p}_{(1 - 1/p)/2}.
\end{equation*}    
\end{lemma}
\begin{proof}

For $x \in \calX$, let us denote  $\calS_x^{(s)} \subset \calX$ the set of states accessible from $x$ in $s$ steps. We have that
\begin{equation*}
\begin{split}
    \sqrt{\calJ_p^{(s)}} &= \sum_{x \in \calX} \pi(x)^{(1 - 1/p)/2} \sum_{x' \in \calX}P^s(x,x')^{(1 - 1/p)/2} \leq \sum_{x \in \calX} \pi(x)^{(1 - 1/p)/2} \sum_{x' \in \calS_x^{(s)}} \left(\frac{1}{\calS_x^{(s)}} \right)^{(1 - 1/p)/2} \\
    &= \sum_{x \in \calX} \pi(x)^{(1 - 1/p)/2}  {\calS_x^{(s)} }^{(1+1/p)/2} \leq \sum_{x \in \calX} \pi(x)^{(1 - 1/p)/2} (s^q G)^{(1 + 1/p)/2}, \\
\end{split}
\end{equation*}
where the first inequality stems from the uniform distribution maximizing the  $\ell_{(1 - 1/p)/2}$ quasi-norm, and the second inequality follows from the polynomial growth assumption.
\end{proof}

\begin{example}[Graph with polynomial growth and stationary distribution with polynomially decaying tail]
We let $\calX = \bbN$, and suppose that there exists $C \in \bbR_+$, and $\alpha \in (1, \infty)$ such that $\pi(x) \leq \frac{C}{x^{1 + \alpha}}$ for any $x \in \calX$. It holds that
\begin{equation*}
    \sum_{x \in \calX} \pi(x)^{(1 - 1/p)/2} \leq C^{(1 - 1/p)/2} \zeta((1 + \alpha)(1 - 1/p)/2),
\end{equation*}
where $\zeta$ is the Riemann Zeta function.

\begin{enumerate}
    \item In the uniformly ergodic setting, Lemma~\ref{lemma:from-polynomial-growth} specializes to
    \begin{equation*}
        \calJ_{\infty, \xi} \leq G C \atmix(\xi)^q \zeta((1+ \alpha)/2)^2,
    \end{equation*}
    and treating $C, G, \alpha$ as constants, an upper bound on the sample complexity for a $(1/4, \delta)$-PAC bound on $\atmix(\xi)$ is given by
    \begin{equation*}
        n \geq \widetilde \bigO \left( \tmix^{q + 1} /\xi^2 \right).
    \end{equation*}
    \item In the non-uniformly setting, we may instead optimize the following product
    \begin{equation*}
        \min_{p \geq 1} \set{ \calB_p G^{1 + 1/p} \atmix(\xi)^{q(1 + 1/p)} \nrm{\pi}^{1 - 1/p}_{(1 - 1/p)/2} }.
    \end{equation*}
\end{enumerate}

\end{example}

\subsection{Finite space}
When the state space $\calX$ is finite and $P$ is supported on $(\calX, \calE)$, the graph-related methods discussed in the previous section remain applicable. 
Alternatively, a cruder---graph agnostic---approach yields the following bound,
\begin{equation*}
\label{eq:from-moderately-large-finite-space}
     \sqrt{\calJ_{\infty}^{(s)}} \leq 
    \sum_{x,x' \in \calX} \left(\frac{1}{\abs{\calX}^2}\right)^{1/2} = \abs{\calX}.
\end{equation*}
Note that this bound matches the heuristic lower bound $\nrm{\pi}_{(1 - 1/p)/2}^{2(1 - 1/p)}$ in the worst-case where $P$ is bistochastic.

\begin{corollary}[of Theorem~\ref{theorem:estimation-average-mixing-time} for $\abs{\calX} < \infty$]
\label{corollary:average-mixing-time-estimation-finite-space}
Let $\calX$ be a finite space.
Let $\eps \in (0,1/4)$, and $\delta, \xi \in (0,1)$.
    Let $n \in \bbN$, and let $\widehat{\tmix}^{\sharp}(\xi) \colon \calX^n \to \bbN$ be the estimator for $\atmix(\xi)$ defined in 
\eqref{eq:estimator-average-tmix}.
For any ergodic stationary Markov chain $X_1, X_2, \dots, X_n$ with transition operator $P \in \calW(\calX)$.
When
\begin{equation*}
        n \geq 1 + 4 \tmix\frac{1 + 2\ceil*{\log_2\frac{1}{\xi(1 - \eps)}}}{\xi^2 \eps^2} \max \set{ \Cue^2 \abs{\calX}^2, 1152 \log \frac{4 \tmix \ceil*{\log_2 1/\xi}}{\delta}},
    \end{equation*}
with probability at least $1 - \delta$, it holds that
\begin{equation*}
\begin{split}
\widehat{\tmix}^{\sharp}(\xi) \in \left[ \tmix^{\sharp}(\xi(1 + \eps)), \tmix^{\sharp}(\xi(1 - \eps)) \right].
\end{split}
\end{equation*}
\end{corollary}

\subsubsection*{Comparison with related work}
\label{section:comparison-relaxation-time-mixing-time-estimation-complexity}

Since the criteria for convergence are different, the problems of estimating the relaxation or worst-case mixing time of a chain to multiplicative accuracy and the problem at hand are not strictly speaking comparable, but it is instructive to inspect the sample complexities.
The problem of estimating the relaxation time of a Markov chain to multiplicative error has sample complexity of the order of $\widetilde\Theta(\tmix/\pi_\star)$ \citep{hsu2019mixing, wolfer2024improved}, while a known upper bound on the complexity of estimating the mixing time is of the order of $\widetilde\bigO(\abs{\calX}\tmix/\pi_\star)$ \citep{wolfer2024empirical}.
Crucially, the above-mentioned problem all suffer a dependency on the minimum stationary probability $\pi_\star$. When $\pi_\star \propto \exp(-\abs{\calX})$, the problem becomes intractable.
Corollary~\ref{corollary:average-mixing-time-estimation-finite-space} yields the superior upper bound
$\widetilde \bigO (\tmix \abs{\calX}^2/\xi^2)$, agnostic of $\pi_\star$ and enjoying an at most linear dependency on the number of parameters of the transition operator.

\subsection{The two-point space}
In this section, we analyze the case where $\abs{\calX} = 2$, illustrating the possibly stark difference between the worst-case mixing time and the average-mixing time.
For $p , q \in (0,1)$, we let
\begin{equation*}
P_{p, q} = \begin{pmatrix}
 1 - p  & p \\
  q & 1 - q
\end{pmatrix}.    
\end{equation*}
Quantities such as $d$ or $d^\sharp$ subscripted with $p,q$ will be understood as pertaining to $P_{p,q}$.
The matrix $P_{p,q}$ is irreducible, and its stationary distribution is readily given by
\begin{equation*}
    \pi_{p,q} \trn = \frac{1}{p+q} \begin{pmatrix} q \\ p \end{pmatrix}.
\end{equation*}
It follows from a direct computation (see e.g. the proof of \citet[Lemma~8.1]{wolfer2024empirical}) that
\begin{equation*}
\begin{split}
P_{p, q}^t
&= 1 \trn \pi_{p,q} + \frac{(1 - p - q)^t}{p + q} \begin{pmatrix} p & -p \\
-q  & q \end{pmatrix}, \\
\end{split}
\end{equation*}
and as a result,
\begin{equation*}
\begin{split}
    d_{p,q}^{\sharp}(t) &= 2 \frac{pq}{(p+q)^2} \abs{1 - p - q}^t, \qquad  d_{p,q}(t) =  \frac{p \lor q}{p+q} \abs{1 - p - q}^t.
\end{split}
\end{equation*}
When $p = q$, the stationary distribution is uniform, $d_{p,q}^{\sharp}(t) = d_{p,q}(t)$, and $
    \tmix(\xi) = \tmix^\sharp(\xi)$.
However, for $\eta \eqdef p/q < 1$, the ratio of the distances is a function of the parameter $\eta$,
\begin{equation*}
    \frac{d_{p,q}(t)}{d^\sharp_{p,q}(t)} = \frac{1 + \eta}{2 \eta} \eqdef r_{\eta}(t),
\end{equation*}
whence the following lemma holds.

\begin{lemma}
    \label{lemma:arbitrary-gap-between-average-and-worst-case-mixing-time}
    There exists a family of Markov chains $$\set{P_{p,q} \colon (p, q) \in (0,1)^2} \in \calW^\star(\set{0,1})$$ such that for any $\xi \in (0,1)$ and any $M > 0$, there exist $p,q \in (0,1)$ such that $
    \tmix(\xi)  > M \tmix^{\sharp}(\xi).$
\end{lemma}

\section{Conclusion}
\label{section:conclusion}
To assess convergence in Markov chains from a single trajectory of observations, recent literature has focused on estimating the worst-case mixing time or the spectral gaps associated with the transition operator.
These mixing parameters lead to concentration inequalities that are almost as powerful as those in the iid setting. 
However, their existence requires strong assumptions on the chain such as uniform or geometric ergodicity, and they are notoriously hard to estimate from the data, with known lower bounds inversely proportional to the minimum stationary probability $\pi_\star$.
As a result, the estimation problem is only statistically tractable for moderately small space sizes and when the tail of the stationary distribution does not decay too rapidly.
In this paper, we relaxed these necessary conditions and limitations by considering the average-mixing time instead. 
Over finite spaces, our analysis showed that the average-mixing time is easier to estimate than its worst-case counterpart, and we could obtain an upper bound on the sample complexity of the problem that is independent of the rate of decay of the stationary distribution.
Moreover, in contradistinction to worst-case mixing time estimation, our analysis extends to countable state spaces and sub-geometrically $\beta$-mixing Markov chains under natural structural assumptions. The trade-off for accessing an easier convergence estimation rate is a logarithmic degradation in the available concentration inequalities.

\appendix

\section{Extension to non-stationary Markov chains}
\label{section:extension-non-stationary}

Although our results are generally stated assuming a stationary start, it is possible to generalize them to non-stationary chains as follows.
Let $\mu \in \calP(\calX)$ absolutely continuous with respect to $\pi$.
Inspired by the proof of \citet[Proposition~3.10]{paulin2015concentration}, an application of H\"{o}lder's inequality with $q,r \geq 1$ with $1/q + 1/r = 1$ yields,
\begin{equation*}
    \bbE_{\mu} f(X_1, \dots, X_n) \leq \nrm{\mu(X)/\pi(X)}_{\pi, q} \nrm{f(X_1, \dots, X_n)}_{\pi, r},
\end{equation*}
and in particular, for any $\eps > 0$, we can bound the probability of an $\eps$-deviation assuming a non-stationary start as
\begin{equation*}
    \PR[\mu]{f(X_1, \dots, X_n) > \eps} \leq \nrm{\mu(X)/\pi(X)}_{\pi, q} \PR[\pi] {f(X_1, \dots, X_n) > \eps}^{1/r}.
\end{equation*}
Setting $q = r = 2$ recovers \citet[Proposition~3.10]{paulin2015concentration}.
The reader is invited to consult \citep[Section~3.3]{paulin2015concentration} for additional methods involving ``burning'' a prefix of the observed trajectory.

\section{Connection with \texorpdfstring{$\beta$}{・趣ｽｲ}-mixing}
\label{section:connection-beta-mixing}
 Recall that for a process $\{X_t\}_{t \in \bbZ}$ on $\calX$, for any $s \in \bbN \cup \{0\} $, the $\beta$-mixing coefficient \citep{bradley2005basic, doukhan2012mixing} is defined by
 \begin{equation*}
     \beta(s) \eqdef \sup_{r \in \bbZ} \beta \left( \sigma\left( \set{X_t \colon t \leq r} \right), \sigma\left( \set{X_t \colon t \geq r + s} \right) \right),
 \end{equation*}
 where $\sigma(\{X_t\}_t)$ is the $\sigma$-field generated by the random variables $\{X_t\}_t$, and for two $\sigma$-fields $\calA$ and $\calB$,
 \begin{equation*}
     \beta(\calA, \calB) \eqdef \sup \frac{1}{2} \sum_{i = 1}^{I} \sum_{j = 1}^{J} \abs{\bbP(A_i \cap B_j) - \bbP(A_i)\bbP(B_j)},
 \end{equation*}
where the supremum is taken over all pairs of finite partitions $\{ A_1, \dots, A_I \}$
and $\{ B_1, \dots, B_J \}$ such that for any $1 \leq i \leq I$ and any $1 \leq j \leq J$, $A_i \in \calA$  and $B_j \in \calB$ \citep{bradley2005basic}.
The random process $\{X_t\}_{t \in \bbZ}$ is called $\beta$-mixing, or absolutely regular, when $$\lim_{s \to \infty} \beta(s) = 0.$$
    In particular, a stationary countable-state Markov chain $X$ is $\beta$-mixing if and only if it is irreducible and aperiodic \citep[Theorem~3.2]{bradley2005basic}.
    
\begin{lemma}[{\citet[Proposition~1]{davydov1974mixing}}]
We let $\{X_t\}_{t \in \bbN}$ be a countable-state Markov chain with transition operator $P$.
Writing $\mu^{(0)}$ for the initial distribution, and for $t > 1$, $\mu^{(t)} = \mu^{(0)} P^t$ for the marginal distribution after time step $t$, it holds that
    \begin{equation*}
        \beta(s) = \sup_{t \in \bbN} \sum_{x \in \calX} \mu^{(t)}(x) \tv{ e_x P^s - \mu^{(s + t)} }. \\
    \end{equation*}
\end{lemma}
It follows from the above lemma that when the chain is stationary, $\mu^{(0)} = \pi = \mu^{(t)}$ for any $t \in \bbN$, and
\begin{equation*}
\begin{split}
    \beta(s) = \sum_{x \in \calX} \pi(x) \tv{ e_x P^s - \pi }, \\
\end{split}
\end{equation*}
thus the average-mixing time is essentially the ``stationary $\beta$-mixing time'' of the chain\footnote{
Similarly, the worst-case mixing time is directly connected to the notion of $\phi$-mixing \citep{davydov1974mixing}.
}. 
Note that with the convention $P^0 = I$, we obtain 
$$\beta(0) = 1 - \nrm{\pi}_2^2,$$
which depends solely on the stationary properties of the chain.

\section{Proofs}
\label{section:proofs}
In this section, we compile technical proofs that were either outlined briefly or omitted in the manuscript.
The subsequent technical tool will prove to be convenient.

\begin{lemma}
\label{lemma:approximate-transcendental-inequality-solver}
    Let $A, B \in \bbR_+$, with $A, B \geq e$.
    \begin{equation*}
       B \geq 2 A \log A \implies B \geq A \log B.
    \end{equation*}
\end{lemma}
\begin{proof}
The function $x \mapsto x / \log x$ is increasing on $[e, \infty)$ and in this range, it holds that $x / \log x \geq e > 2$.
Hence, for $B \geq 2 A \log A$,
it holds that
\begin{equation*}
    \frac{B}{\log B} \geq \frac{2 A \log A}{\log A + \log (2 \log A)} \geq \frac{2 A \log A}{\log A + \log A} = A,
\end{equation*}
whence the lemma holds.
\end{proof}

\subsection{Proof of Lemma~\ref{lemma:large-deviation-bound-average-mixing-time}}
\label{proof-large-deviation-bound-average-mixing-time}
The proof is standard, and based on the blocking technique credited to Bernstein and further developed in the series of papers \citep{yu1994rates, eberlein1984weak, volkonskii1959some}.
For simplicity, let us first assume that $n = 2Bs$, where $B, s \in \bbN$.
For $b \in [B]$, we write
\begin{equation*}
\begin{split}
    X^{[2b]} &= X^{[2b]}_1, \dots, X^{[2b]}_s = X_{(2b - 1)s + 1}, \dots, X_{2bs}, \\
    X^{[2b - 1]} &= X^{[2b - 1]}_1, \dots, X^{[2b - 1]}_s = X_{(2b - 2)s + 1}, \dots, X_{(2b - 1)s}, \\
\end{split}
\end{equation*}
and consider the partitioned process
\begin{equation*}
   X = X^{[1]}, X^{[2]}, \dots, X^{[2B]}. 
\end{equation*}
In other words, we decompose the observed trajectory into $B$ even and $B$ odd blocks of size $s$, as illustrated in Figure~\ref{fig:blocking-method}. 
\begin{figure}
    \centering

\tikzset{
pattern size/.store in=\mcSize, 
pattern size = 5pt,
pattern thickness/.store in=\mcThickness, 
pattern thickness = 0.3pt,
pattern radius/.store in=\mcRadius, 
pattern radius = 1pt}
\makeatletter
\pgfutil@ifundefined{pgf@pattern@name@_p4s6kzw11}{
\pgfdeclarepatternformonly[\mcThickness,\mcSize]{_p4s6kzw11}
{\pgfqpoint{0pt}{0pt}}
{\pgfpoint{\mcSize+\mcThickness}{\mcSize+\mcThickness}}
{\pgfpoint{\mcSize}{\mcSize}}
{
\pgfsetcolor{\tikz@pattern@color}
\pgfsetlinewidth{\mcThickness}
\pgfpathmoveto{\pgfqpoint{0pt}{0pt}}
\pgfpathlineto{\pgfpoint{\mcSize+\mcThickness}{\mcSize+\mcThickness}}
\pgfusepath{stroke}
}}
\makeatother

\tikzset{
pattern size/.store in=\mcSize, 
pattern size = 5pt,
pattern thickness/.store in=\mcThickness, 
pattern thickness = 0.3pt,
pattern radius/.store in=\mcRadius, 
pattern radius = 1pt}
\makeatletter
\pgfutil@ifundefined{pgf@pattern@name@_y21xws0f4}{
\pgfdeclarepatternformonly[\mcThickness,\mcSize]{_y21xws0f4}
{\pgfqpoint{0pt}{0pt}}
{\pgfpoint{\mcSize+\mcThickness}{\mcSize+\mcThickness}}
{\pgfpoint{\mcSize}{\mcSize}}
{
\pgfsetcolor{\tikz@pattern@color}
\pgfsetlinewidth{\mcThickness}
\pgfpathmoveto{\pgfqpoint{0pt}{0pt}}
\pgfpathlineto{\pgfpoint{\mcSize+\mcThickness}{\mcSize+\mcThickness}}
\pgfusepath{stroke}
}}
\makeatother

\tikzset{
pattern size/.store in=\mcSize, 
pattern size = 5pt,
pattern thickness/.store in=\mcThickness, 
pattern thickness = 0.3pt,
pattern radius/.store in=\mcRadius, 
pattern radius = 1pt}
\makeatletter
\pgfutil@ifundefined{pgf@pattern@name@_maj9v1h8s}{
\pgfdeclarepatternformonly[\mcThickness,\mcSize]{_maj9v1h8s}
{\pgfqpoint{0pt}{0pt}}
{\pgfpoint{\mcSize+\mcThickness}{\mcSize+\mcThickness}}
{\pgfpoint{\mcSize}{\mcSize}}
{
\pgfsetcolor{\tikz@pattern@color}
\pgfsetlinewidth{\mcThickness}
\pgfpathmoveto{\pgfqpoint{0pt}{0pt}}
\pgfpathlineto{\pgfpoint{\mcSize+\mcThickness}{\mcSize+\mcThickness}}
\pgfusepath{stroke}
}}
\makeatother

\tikzset{
pattern size/.store in=\mcSize, 
pattern size = 5pt,
pattern thickness/.store in=\mcThickness, 
pattern thickness = 0.3pt,
pattern radius/.store in=\mcRadius, 
pattern radius = 1pt}
\makeatletter
\pgfutil@ifundefined{pgf@pattern@name@_2kex6mik8}{
\pgfdeclarepatternformonly[\mcThickness,\mcSize]{_2kex6mik8}
{\pgfqpoint{0pt}{0pt}}
{\pgfpoint{\mcSize+\mcThickness}{\mcSize+\mcThickness}}
{\pgfpoint{\mcSize}{\mcSize}}
{
\pgfsetcolor{\tikz@pattern@color}
\pgfsetlinewidth{\mcThickness}
\pgfpathmoveto{\pgfqpoint{0pt}{0pt}}
\pgfpathlineto{\pgfpoint{\mcSize+\mcThickness}{\mcSize+\mcThickness}}
\pgfusepath{stroke}
}}
\makeatother

\tikzset{
pattern size/.store in=\mcSize, 
pattern size = 5pt,
pattern thickness/.store in=\mcThickness, 
pattern thickness = 0.3pt,
pattern radius/.store in=\mcRadius, 
pattern radius = 1pt}
\makeatletter
\pgfutil@ifundefined{pgf@pattern@name@_tf3g6m8z9}{
\pgfdeclarepatternformonly[\mcThickness,\mcSize]{_tf3g6m8z9}
{\pgfqpoint{0pt}{0pt}}
{\pgfpoint{\mcSize+\mcThickness}{\mcSize+\mcThickness}}
{\pgfpoint{\mcSize}{\mcSize}}
{
\pgfsetcolor{\tikz@pattern@color}
\pgfsetlinewidth{\mcThickness}
\pgfpathmoveto{\pgfqpoint{0pt}{0pt}}
\pgfpathlineto{\pgfpoint{\mcSize+\mcThickness}{\mcSize+\mcThickness}}
\pgfusepath{stroke}
}}
\makeatother

\tikzset{
pattern size/.store in=\mcSize, 
pattern size = 5pt,
pattern thickness/.store in=\mcThickness, 
pattern thickness = 0.3pt,
pattern radius/.store in=\mcRadius, 
pattern radius = 1pt}
\makeatletter
\pgfutil@ifundefined{pgf@pattern@name@_ncx6ofe92}{
\pgfdeclarepatternformonly[\mcThickness,\mcSize]{_ncx6ofe92}
{\pgfqpoint{0pt}{0pt}}
{\pgfpoint{\mcSize+\mcThickness}{\mcSize+\mcThickness}}
{\pgfpoint{\mcSize}{\mcSize}}
{
\pgfsetcolor{\tikz@pattern@color}
\pgfsetlinewidth{\mcThickness}
\pgfpathmoveto{\pgfqpoint{0pt}{0pt}}
\pgfpathlineto{\pgfpoint{\mcSize+\mcThickness}{\mcSize+\mcThickness}}
\pgfusepath{stroke}
}}
\makeatother

\tikzset{
pattern size/.store in=\mcSize, 
pattern size = 5pt,
pattern thickness/.store in=\mcThickness, 
pattern thickness = 0.3pt,
pattern radius/.store in=\mcRadius, 
pattern radius = 1pt}
\makeatletter
\pgfutil@ifundefined{pgf@pattern@name@_cmfop4xza}{
\pgfdeclarepatternformonly[\mcThickness,\mcSize]{_cmfop4xza}
{\pgfqpoint{0pt}{0pt}}
{\pgfpoint{\mcSize+\mcThickness}{\mcSize+\mcThickness}}
{\pgfpoint{\mcSize}{\mcSize}}
{
\pgfsetcolor{\tikz@pattern@color}
\pgfsetlinewidth{\mcThickness}
\pgfpathmoveto{\pgfqpoint{0pt}{0pt}}
\pgfpathlineto{\pgfpoint{\mcSize+\mcThickness}{\mcSize+\mcThickness}}
\pgfusepath{stroke}
}}
\makeatother

\tikzset{
pattern size/.store in=\mcSize, 
pattern size = 5pt,
pattern thickness/.store in=\mcThickness, 
pattern thickness = 0.3pt,
pattern radius/.store in=\mcRadius, 
pattern radius = 1pt}
\makeatletter
\pgfutil@ifundefined{pgf@pattern@name@_ews9cmrxv}{
\pgfdeclarepatternformonly[\mcThickness,\mcSize]{_ews9cmrxv}
{\pgfqpoint{0pt}{0pt}}
{\pgfpoint{\mcSize+\mcThickness}{\mcSize+\mcThickness}}
{\pgfpoint{\mcSize}{\mcSize}}
{
\pgfsetcolor{\tikz@pattern@color}
\pgfsetlinewidth{\mcThickness}
\pgfpathmoveto{\pgfqpoint{0pt}{0pt}}
\pgfpathlineto{\pgfpoint{\mcSize+\mcThickness}{\mcSize+\mcThickness}}
\pgfusepath{stroke}
}}
\makeatother

\tikzset{
pattern size/.store in=\mcSize, 
pattern size = 5pt,
pattern thickness/.store in=\mcThickness, 
pattern thickness = 0.3pt,
pattern radius/.store in=\mcRadius, 
pattern radius = 1pt}
\makeatletter
\pgfutil@ifundefined{pgf@pattern@name@_sxhrdnfa0}{
\pgfdeclarepatternformonly[\mcThickness,\mcSize]{_sxhrdnfa0}
{\pgfqpoint{0pt}{-\mcThickness}}
{\pgfpoint{\mcSize}{\mcSize}}
{\pgfpoint{\mcSize}{\mcSize}}
{
\pgfsetcolor{\tikz@pattern@color}
\pgfsetlinewidth{\mcThickness}
\pgfpathmoveto{\pgfqpoint{0pt}{\mcSize}}
\pgfpathlineto{\pgfpoint{\mcSize+\mcThickness}{-\mcThickness}}
\pgfusepath{stroke}
}}
\makeatother

\tikzset{
pattern size/.store in=\mcSize, 
pattern size = 5pt,
pattern thickness/.store in=\mcThickness, 
pattern thickness = 0.3pt,
pattern radius/.store in=\mcRadius, 
pattern radius = 1pt}
\makeatletter
\pgfutil@ifundefined{pgf@pattern@name@_40t7md2zg}{
\pgfdeclarepatternformonly[\mcThickness,\mcSize]{_40t7md2zg}
{\pgfqpoint{0pt}{-\mcThickness}}
{\pgfpoint{\mcSize}{\mcSize}}
{\pgfpoint{\mcSize}{\mcSize}}
{
\pgfsetcolor{\tikz@pattern@color}
\pgfsetlinewidth{\mcThickness}
\pgfpathmoveto{\pgfqpoint{0pt}{\mcSize}}
\pgfpathlineto{\pgfpoint{\mcSize+\mcThickness}{-\mcThickness}}
\pgfusepath{stroke}
}}
\makeatother

\tikzset{
pattern size/.store in=\mcSize, 
pattern size = 5pt,
pattern thickness/.store in=\mcThickness, 
pattern thickness = 0.3pt,
pattern radius/.store in=\mcRadius, 
pattern radius = 1pt}
\makeatletter
\pgfutil@ifundefined{pgf@pattern@name@_1u56gfhmg}{
\pgfdeclarepatternformonly[\mcThickness,\mcSize]{_1u56gfhmg}
{\pgfqpoint{0pt}{-\mcThickness}}
{\pgfpoint{\mcSize}{\mcSize}}
{\pgfpoint{\mcSize}{\mcSize}}
{
\pgfsetcolor{\tikz@pattern@color}
\pgfsetlinewidth{\mcThickness}
\pgfpathmoveto{\pgfqpoint{0pt}{\mcSize}}
\pgfpathlineto{\pgfpoint{\mcSize+\mcThickness}{-\mcThickness}}
\pgfusepath{stroke}
}}
\makeatother

\tikzset{
pattern size/.store in=\mcSize, 
pattern size = 5pt,
pattern thickness/.store in=\mcThickness, 
pattern thickness = 0.3pt,
pattern radius/.store in=\mcRadius, 
pattern radius = 1pt}
\makeatletter
\pgfutil@ifundefined{pgf@pattern@name@_yteobkg5r}{
\pgfdeclarepatternformonly[\mcThickness,\mcSize]{_yteobkg5r}
{\pgfqpoint{0pt}{-\mcThickness}}
{\pgfpoint{\mcSize}{\mcSize}}
{\pgfpoint{\mcSize}{\mcSize}}
{
\pgfsetcolor{\tikz@pattern@color}
\pgfsetlinewidth{\mcThickness}
\pgfpathmoveto{\pgfqpoint{0pt}{\mcSize}}
\pgfpathlineto{\pgfpoint{\mcSize+\mcThickness}{-\mcThickness}}
\pgfusepath{stroke}
}}
\makeatother

\tikzset{
pattern size/.store in=\mcSize, 
pattern size = 5pt,
pattern thickness/.store in=\mcThickness, 
pattern thickness = 0.3pt,
pattern radius/.store in=\mcRadius, 
pattern radius = 1pt}
\makeatletter
\pgfutil@ifundefined{pgf@pattern@name@_1sf1jxqdc}{
\pgfdeclarepatternformonly[\mcThickness,\mcSize]{_1sf1jxqdc}
{\pgfqpoint{0pt}{-\mcThickness}}
{\pgfpoint{\mcSize}{\mcSize}}
{\pgfpoint{\mcSize}{\mcSize}}
{
\pgfsetcolor{\tikz@pattern@color}
\pgfsetlinewidth{\mcThickness}
\pgfpathmoveto{\pgfqpoint{0pt}{\mcSize}}
\pgfpathlineto{\pgfpoint{\mcSize+\mcThickness}{-\mcThickness}}
\pgfusepath{stroke}
}}
\makeatother

\tikzset{
pattern size/.store in=\mcSize, 
pattern size = 5pt,
pattern thickness/.store in=\mcThickness, 
pattern thickness = 0.3pt,
pattern radius/.store in=\mcRadius, 
pattern radius = 1pt}
\makeatletter
\pgfutil@ifundefined{pgf@pattern@name@_otxor0l8h}{
\pgfdeclarepatternformonly[\mcThickness,\mcSize]{_otxor0l8h}
{\pgfqpoint{0pt}{-\mcThickness}}
{\pgfpoint{\mcSize}{\mcSize}}
{\pgfpoint{\mcSize}{\mcSize}}
{
\pgfsetcolor{\tikz@pattern@color}
\pgfsetlinewidth{\mcThickness}
\pgfpathmoveto{\pgfqpoint{0pt}{\mcSize}}
\pgfpathlineto{\pgfpoint{\mcSize+\mcThickness}{-\mcThickness}}
\pgfusepath{stroke}
}}
\makeatother

\tikzset{
pattern size/.store in=\mcSize, 
pattern size = 5pt,
pattern thickness/.store in=\mcThickness, 
pattern thickness = 0.3pt,
pattern radius/.store in=\mcRadius, 
pattern radius = 1pt}
\makeatletter
\pgfutil@ifundefined{pgf@pattern@name@_j2nrr9l9f}{
\pgfdeclarepatternformonly[\mcThickness,\mcSize]{_j2nrr9l9f}
{\pgfqpoint{0pt}{-\mcThickness}}
{\pgfpoint{\mcSize}{\mcSize}}
{\pgfpoint{\mcSize}{\mcSize}}
{
\pgfsetcolor{\tikz@pattern@color}
\pgfsetlinewidth{\mcThickness}
\pgfpathmoveto{\pgfqpoint{0pt}{\mcSize}}
\pgfpathlineto{\pgfpoint{\mcSize+\mcThickness}{-\mcThickness}}
\pgfusepath{stroke}
}}
\makeatother

\tikzset{
pattern size/.store in=\mcSize, 
pattern size = 5pt,
pattern thickness/.store in=\mcThickness, 
pattern thickness = 0.3pt,
pattern radius/.store in=\mcRadius, 
pattern radius = 1pt}
\makeatletter
\pgfutil@ifundefined{pgf@pattern@name@_b3q8sgbb5}{
\pgfdeclarepatternformonly[\mcThickness,\mcSize]{_b3q8sgbb5}
{\pgfqpoint{0pt}{-\mcThickness}}
{\pgfpoint{\mcSize}{\mcSize}}
{\pgfpoint{\mcSize}{\mcSize}}
{
\pgfsetcolor{\tikz@pattern@color}
\pgfsetlinewidth{\mcThickness}
\pgfpathmoveto{\pgfqpoint{0pt}{\mcSize}}
\pgfpathlineto{\pgfpoint{\mcSize+\mcThickness}{-\mcThickness}}
\pgfusepath{stroke}
}}
\makeatother
\tikzset{every picture/.style={line width=0.75pt}} %

\begin{tikzpicture}[x=1.2pt,y=1.2pt,yscale=-1,xscale=1]

\draw  [draw opacity=0] (120,81) -- (440,81) -- (440,101) -- (120,101) -- cycle ; \draw   (140,81) -- (140,101)(160,81) -- (160,101)(180,81) -- (180,101)(200,81) -- (200,101)(220,81) -- (220,101)(240,81) -- (240,101)(260,81) -- (260,101)(280,81) -- (280,101)(300,81) -- (300,101)(320,81) -- (320,101)(340,81) -- (340,101)(360,81) -- (360,101)(380,81) -- (380,101)(400,81) -- (400,101)(420,81) -- (420,101) ; \draw    ; \draw   (120,81) -- (440,81) -- (440,101) -- (120,101) -- cycle ;
\draw  [pattern=_p4s6kzw11,pattern size=3pt,pattern thickness=0.75pt,pattern radius=0pt, pattern color=white] (120,81) -- (140,81) -- (140,101) -- (120,101) -- cycle ;
\draw  [pattern=_y21xws0f4,pattern size=3pt,pattern thickness=0.75pt,pattern radius=0pt, pattern color=white] (160,81) -- (180,81) -- (180,101) -- (160,101) -- cycle ;
\draw  [pattern=_maj9v1h8s,pattern size=3pt,pattern thickness=0.75pt,pattern radius=0pt, pattern color=white] (200,81) -- (220,81) -- (220,101) -- (200,101) -- cycle ;
\draw  [pattern=_2kex6mik8,pattern size=3pt,pattern thickness=0.75pt,pattern radius=0pt, pattern color=white] (240,81) -- (260,81) -- (260,101) -- (240,101) -- cycle ;
\draw  [pattern=_tf3g6m8z9,pattern size=3pt,pattern thickness=0.75pt,pattern radius=0pt, pattern color=white] (280,81) -- (300,81) -- (300,101) -- (280,101) -- cycle ;
\draw  [pattern=_ncx6ofe92,pattern size=3pt,pattern thickness=0.75pt,pattern radius=0pt, pattern color=white] (320,81) -- (340,81) -- (340,101) -- (320,101) -- cycle ;
\draw  [pattern=_cmfop4xza,pattern size=3pt,pattern thickness=0.75pt,pattern radius=0pt, pattern color=white] (360,81) -- (380,81) -- (380,101) -- (360,101) -- cycle ;
\draw  [pattern=_ews9cmrxv,pattern size=3pt,pattern thickness=0.75pt,pattern radius=0pt, pattern color=white] (400,81) -- (420,81) -- (420,101) -- (400,101) -- cycle ;

\draw  [pattern=_sxhrdnfa0,pattern size=3pt,pattern thickness=0.75pt,pattern radius=0pt, pattern color=wred] (140,81) -- (160,81) -- (160,101) -- (140,101) -- cycle ;
\draw  [pattern=_40t7md2zg,pattern size=3pt,pattern thickness=0.75pt,pattern radius=0pt, pattern color=wred] (180,81) -- (200,81) -- (200,101) -- (180,101) -- cycle ;
\draw  [pattern=_1u56gfhmg,pattern size=3pt,pattern thickness=0.75pt,pattern radius=0pt, pattern color=wred] (220,81) -- (240,81) -- (240,101) -- (220,101) -- cycle ;
\draw  [pattern=_yteobkg5r,pattern size=3pt,pattern thickness=0.75pt,pattern radius=0pt, pattern color=wred] (260,81) -- (280,81) -- (280,101) -- (260,101) -- cycle ;
\draw  [pattern=_1sf1jxqdc,pattern size=3pt,pattern thickness=0.75pt,pattern radius=0pt, pattern color=wred] (300,81) -- (320,81) -- (320,101) -- (300,101) -- cycle ;
\draw  [pattern=_otxor0l8h,pattern size=3pt,pattern thickness=0.75pt,pattern radius=0pt, pattern color=wred] (340,81) -- (360,81) -- (360,101) -- (340,101) -- cycle ;
\draw  [pattern=_j2nrr9l9f,pattern size=3pt,pattern thickness=0.75pt,pattern radius=0pt, pattern color=wred] (380,81) -- (400,81) -- (400,101) -- (380,101) -- cycle ;
\draw  [pattern=_b3q8sgbb5,pattern size=3pt,pattern thickness=0.75pt,pattern radius=0pt, pattern color=wred] (420,81) -- (440,81) -- (440,101) -- (420,101) -- cycle ;

\draw [color={rgb, 255:red, 155; green, 155; blue, 155 }  ,draw opacity=1 ]   (140,120) -- (160,120) ;
\draw [shift={(160,120)}, rotate = 180] [color={rgb, 255:red, 155; green, 155; blue, 155 }  ,draw opacity=1 ][line width=0.75]    (0,5.59) -- (0,-5.59)   ;
\draw [shift={(140,120)}, rotate = 180] [color={rgb, 255:red, 155; green, 155; blue, 155 }  ,draw opacity=1 ][line width=0.75]    (0,5.59) -- (0,-5.59)   ;
\draw [color={rgb, 255:red, 155; green, 155; blue, 155 }  ,draw opacity=1 ]   (120,71) -- (440,71) ;
\draw [shift={(440,71)}, rotate = 180] [color={rgb, 255:red, 155; green, 155; blue, 155 }  ,draw opacity=1 ][line width=0.75]    (0,5.59) -- (0,-5.59)   ;
\draw [shift={(120,71)}, rotate = 180] [color={rgb, 255:red, 155; green, 155; blue, 155 }  ,draw opacity=1 ][line width=0.75]    (0,5.59) -- (0,-5.59)   ;

\draw (147,123) node [anchor=north west][inner sep=0.75pt]   [align=left] {$s$};
\draw (259,60) node [anchor=north west][inner sep=0.75pt]   [align=left] {$n = 2Bs$};

\draw (125,102) node [anchor=north west][inner sep=0.75pt]   [align=left] {$X^{[1]}$};

\draw (145,102) node [anchor=north west][inner sep=0.75pt]   [align=left] {$X^{[2]}$};

\draw (165,102) node [anchor=north west][inner sep=0.75pt]   [align=left] {$X^{[3]}$};

\draw (185,102) node [anchor=north west][inner sep=0.75pt]   [align=left] {$X^{[4]}$};

\draw (205,106) node [anchor=north west][inner sep=0.75pt]   [align=left] {$\cdots$};

\draw (405,106) node [anchor=north west][inner sep=0.75pt]   [align=left] {$\cdots$};

\draw (420,102) node [anchor=north west][inner sep=0.75pt]   [align=left] {$X^{[2B]}$};

\end{tikzpicture}
    \caption{Decomposing the process into Bernstein blocks.}
    \label{fig:blocking-method}
\end{figure}
Invoking Hoeffding's inequality \citep{hoeffding1963probability} and \citet[Corollary~2.7]{yu1994rates}, it holds that
\begin{equation}
\label{eq:decomposition-blocking-and-concentration}
\begin{split}
    \PR{ \frac{1}{n} \sum_{t = 1}^n f(X_t) > \eps } &\leq \sum_{\sigma \in \set{0,1}} \PR{ \sum_{b = 1}^{B} \sum_{t = 1}^s f\left(X_t^{[2b - \sigma]}\right) > n \eps/2 } \\ &\leq 2 \exp\left( -\frac{n \eps^2}{4 s} \right) + 2 \left(B - 1\right) \beta(s).
\end{split}
\end{equation}
We first prove the claim in the sub-exponential $\beta$-mixing setting, that is for $\beta(s) \leq \beta_ 0 e^{- \beta_1 s^b}$ with parameters $b, \beta_0, \beta_1$ in their specified range.
We may assume without loss of generality that $\beta_1 \leq 1$.
Let $\overline{\beta}(s) = \beta_0 \exp \left( - \beta_1 s^b \right)$ and let $\abtmix$ the average-mixing time associated with this envelope.
For any $\xi \in (0,1)$, it holds that,
\begin{equation*}
    \abtmix(\xi) = \ceil*{\left(\frac{1}{\beta_1} \log \frac{\beta_0}{\xi} \right)^{1/b}}.
\end{equation*}
Observing that
\begin{equation*}
    2(B - 1) \beta(s) \leq \frac{n \beta_0}{s} \exp(-\beta_1 s^b),
\end{equation*}
and denoting $W_0$ the principal branch of the product logarithm Lambert function \citep{corless1996lambert},
for
\begin{equation*}
   s = \ceil*{ \left( \frac{1}{b \beta_1} W_0 \left( b \beta_1 \left[ \frac{2n \beta_0}{\delta} \right]^{b}\right) \right)^{1/b} },
\end{equation*}
it holds that $2(B - 1) \beta(s) \leq \delta/2$.
As a consequence, the deviation probability at \eqref{eq:decomposition-blocking-and-concentration} is bounded by $\delta$ when the trajectory length $n$ satisfies the inequality,
\begin{equation*}
    n \geq \frac{4}{\eps^2} \ceil*{ \left( \frac{1}{b \beta_1} W_0 \left( b \beta_1 \left[ \frac{2n \beta_0}{\delta} \right]^{b}\right) \right)^{1/b} } \log \left(\frac{4}{\delta} \right).
\end{equation*}
We now give an approximate albeit more tractable sufficient condition on $n$.
Since $b \beta_1 \leq 1$, 
for the above argument of $W_0$ to be no smaller than $e$, it is enough to have,
\begin{equation*}
    n \geq \frac{\delta}{2 \beta_0} \left( \frac{e}{b \beta_1}\right)^{1/b}.
\end{equation*}
Since $\beta_0 \geq 1, b \leq 1$ and $x \leq \frac{1}{e} \implies x \leq \log \frac{1}{x}$,
\begin{equation*}
    \frac{\delta}{2 \beta_0} \left( \frac{e}{b \beta_1}\right)^{1/b} \leq \frac{e}{2} \log \left( \frac{e \beta_0}{\delta} \right) \left( \frac{e}{b \beta_1}\right)^{1/b} \leq \frac{e}{2}  \left( \frac{e}{b \beta_1} \log \left( \frac{e \beta_0}{\delta} \right)
 \right)^{1/b} \leq \frac{e}{2} \left( \frac{e}{b}\right)^{1/b} \abtmix(\delta/e).
\end{equation*}
Therefore,
\begin{equation*}
    n \geq \frac{e}{2} \left( \frac{e}{b}\right)^{1/b} \abtmix(\delta/e),
\end{equation*}
is sufficient.
Further relying on the upper bounds 
$$x \geq e \implies W_0(x) \leq \log(x) \;\text{and}\; x \geq 1 \implies \ceil{x} \leq 2x,$$ it suffices for $n$ to satisfy
\begin{equation*}
    n^b \geq \left( \frac{8}{\eps^{2}} \log \left( \frac{4}{\delta} \right) \right)^{b} \frac{1}{b \beta_1} \log \left(  b\beta_1 \left[\frac{2 n \beta_0}{\delta}  \right]^b \right) .
\end{equation*}
Setting 
\begin{equation*}
    A = \left[ \frac{16 \beta_0}{\delta \eps^2} \log \left(\frac{4}{\delta}\right) \right]^{b}, \qquad B = b \beta_1 \left[ \frac{2n \beta_0}{\delta} \right]^b,
\end{equation*}
and invoking Lemma~\ref{lemma:approximate-transcendental-inequality-solver},
it is enough that
\begin{equation*}
    n \geq \frac{8}{\eps^2} \log \left( \frac{4}{\delta}\right)  2^{1/b} \abtmix\left( \xi(\eps, \delta) \right),
\end{equation*}
with $\xi(\eps, \delta) = \frac{\delta \eps^2}{16 \log \left(\frac{4}{\delta}\right)}$.
In order to streamline this upper bound, observe that, since
$\beta_0 \geq 1$, $\beta_1 \leq 1$, and $\eps,\delta\in(0,1)$,
\begin{equation*}
\begin{split}
    \frac{1}{\beta_1}
    \log\left(
        \frac{\beta_0}{\xi(\eps,\delta)}
    \right)
    &=
    \frac{1}{\beta_1}
    \left[
        \log\left(\frac{e\beta_0}{\delta}\right)
        +
        \log\left(
            \frac{16\log(4/\delta)}
                 {e\eps^2}
        \right)
    \right] \\
    &\leq
    \frac{1}{\beta_1}
    \log\left(\frac{e\beta_0}{\delta}\right)
    \left[
        1+
        \log\left(
            \frac{64\log(4/\delta)}
                 {e\eps^2}
        \right)
    \right].
\end{split}
\end{equation*}
Here, the constant $64$ is used in anticipation of the final reduction
to the case of a general trajectory length, where $\eps$ is replaced by
$\eps/2$. Using $\ceil{x}\leq2x$ for $x\geq1$, it follows that
\begin{equation*}
\begin{split}
    \abtmix\left(\xi(\eps,\delta)\right)
    &\leq
    2\abtmix\left(\delta/e\right)
    \left[
        1+
        \log\left(
            \frac{64\log(4/\delta)}
                 {e\eps^2}
        \right)
    \right]^{1/b}.
\end{split}
\end{equation*}
Combining this inequality with the preceding sufficient conditions and
adjusting the universal constant, it is enough that
\begin{equation*}
    n
    \geq
    \frac{C}{\eps^2}
    \frac{1}{b}
    \left(\frac{e}{b}\right)^{1/b}
    \abtmix\left(\delta/e\right)
    \log\left(\frac{4}{\delta}\right)
    \left[
        1+
        \log\left(
            \frac{64\log(4/\delta)}
                 {e\eps^2}
        \right)
    \right]^{1/b},
\end{equation*}
for some universal constant $C\in\bbR_+$.
Finally, let us now move on to the polynomial case, that is,
\begin{equation*}
\beta(s) \leq \frac{\beta_1}{s^b},   
\end{equation*}
for some $\beta_1,b>0$.
For any $\xi \in (0,1)$, it holds that
\begin{equation*}
 \abtmix(\xi) = \ceil*{\left(\frac{\beta_1}{\xi} \right)^{1/b}}.
\end{equation*}
Having $\beta_1 \geq 1$ ensures that $\frac{2n \beta_1}{\delta} \geq 1$.
Setting
\begin{equation*}
s  = 
\ceil*{\left(\frac{2 n \beta_1}{\delta} \right)^{\frac{1}{b+1}}},
\end{equation*}
yields
\begin{equation*}
    \frac{n}{s} \beta(s) \leq \frac{n \beta_1}{s^{b+1}} \leq n \beta_1 \left( \frac{\delta}{2n \beta_1} \right)  = \frac{\delta}{2}.
\end{equation*}
Finally, setting 
\begin{equation*}
    n \geq \left( \frac{8}{\eps^2} \log \frac{4}{\delta} \right)^{\frac{b+1}{b}} \left( \frac{2 \beta_1}{\delta} \right)^{\frac{1}{b}}
\end{equation*}
leads to
\begin{equation*}
    2 \exp\left( - \frac{n \eps^2}{4 s} \right) \leq 2 \exp\left( - \frac{n \eps^2}{8 \left(\frac{2 n \beta_1}{\delta} \right)^{\frac{1}{b+1}}} \right) \leq \frac{\delta}{2}.
\end{equation*}
It remains to remove the simplifying assumption $n = 2B s$ in order to prove the claim for general $n$.
Let us define
\begin{equation*}
    n_0 = 2s \floor*{ \frac{n}{2s} },
\end{equation*}
and observe that $n_0 \leq n < n_0 + 2s$. Since by assumption, $f \in [-1, 1]$, it holds that
\begin{equation*}
    \frac{1}{n} \sum_{t=1}^{n} f(X_t) \leq \frac{n_0}{n} \left( \frac{1}{n_0} \sum_{t=1}^{n_0} f(X_t) \right) + \frac{2s}{n}.
\end{equation*}
As a result, when $ n \geq 4s/ \eps$, we have
\begin{equation*}
    \set{ \frac{1}{n} \sum_{t=1}^{n} f(X_t) > \eps } \subset \set{ \frac{1}{n_0} \sum_{t=1}^{n_0} f(X_t) > \eps/2 },
\end{equation*}
and the previous argument can be applied verbatim to the truncated trajectory $n_0$ with $\eps/2$ in lieu of $\eps$, modulo adapting the universal constant.
\qed

\subsection{Proof of Lemma~\ref{lemma:spectral-upper-bound}}
The proof closely follows the techniques in \citet[Lemma~12.2, Theorem~12.4]{levin2009markov}, albeit in a countably infinite alphabet setting. 
The linear operator defined by  $P$ is self-adjoint in $\ell_2(\pi)$, hence normal.
We can write,
\begin{equation*}
\begin{split}
d^{\sharp}(t) &\eqdef \sum_{x \in \calX} \pi(x) \tv{ \delta_x P^t - \pi } \\
&= \frac{1}{2} \sum_{x, x' \in \calX} \pi(x) \abs{ P^t(x,x') - \pi(x') } \\
&= \frac{1}{2} \sum_{x, x' \in \calX} \pi(x) \pi(x') \abs{ \frac{P^t(x,x')}{\pi(x')} - 1 } \\
&\stackrel{(\spadesuit)}{\leq} \frac{1}{2} \sum_{x, x' \in \calX} \pi(x) \pi(x') \frac{(1 - \gamma_\star)^t}{\sqrt{\pi(x)\pi(x')}} \\
&= \frac{1}{2} (1 - \gamma_\star)^t \sum_{x, x' \in \calX} {\sqrt{\pi(x)\pi(x')}} \\
&= \frac{1}{2} (1 - \gamma_\star)^t \nrm{\pi}_{1/2}. \\
\end{split}
\end{equation*}
Solving for a proximity parameter $\xi \in (0,1/2)$,
we obtain that
\begin{equation*}
    \atmix(\xi) \leq \ceil*{\trel \log \frac{\nrm{\pi}_{1/2}}{2 \xi}}.
\end{equation*}
To obtain the inequality ($\spadesuit$), we use compactness of $P$ as follows.
From the spectral theorem, there exists a $\langle \cdot, \cdot \rangle_\pi$-orthonormal basis $(f_i)_{i \geq 1}$ of eigenfunctions of $P$ with real eigenvalues $(\lambda_i)_{i \geq 1}$, and we can take $f_1 = 1, \lambda_1 = 1$, which is simple from ergodicity.
Let $x_0 \in \calX$, define $\delta^{\pi}_{x_0} \eqdef \frac{\delta_{x_0}}{\pi(x_0)}$, and observe that
\begin{equation*}
    \nrm{\delta^{\pi}_{x_0}}_\pi^2 = \sum_{x \in \calX} \delta^{\pi}_{x_0}(x) \delta^{\pi}_{x_0}(x) \pi(x) = \frac{1}{\pi(x_0)},
\end{equation*}
which is finite, thus $\delta^{\pi}_{x_0}  \in \ell_2(\pi)$.
Moreover, note that
\begin{equation*}
    \langle \delta^{\pi}_{x_0}, f_i \rangle_\pi = \sum_{x \in \calX} \delta^{\pi}_{x_0}(x) f_i(x) \pi(x) = f_i(x_0),
\end{equation*}
thus 
applying Parseval identity on $\delta^\pi_{x_0}$ yields
\begin{equation*}
    \nrm{\delta^\pi_{x_0}}_\pi^2 = \sum_{i \geq 1} \abs{\langle \delta^\pi_{x_0}, f_i \rangle_\pi}^2 = \sum_{i \geq 1} f_i(x_0)^2,
\end{equation*}
and 
\begin{equation*}
    \sum_{i \geq 1} f_i(x_0)^2 = \frac{1}{\pi(x_0)}.
\end{equation*}
As a result, from Cauchy-Schwarz inequality, for any $x_0, x_0' \in \calX$, we have
\begin{equation}
\label{eq:convergent-series}
    \sum_{i \geq 1} \abs{f_i(x_0) f_i(x'_0)} \leq \sqrt{ \left( \sum_{i \geq 1} f_i(x_0)^2 \right) \left( \sum_{i \geq 1} f_i(x'_0)^2 \right) } = \frac{1}{\sqrt{\pi(x_0) \pi(x'_0)}},
\end{equation}
which is also finite.
Since $\delta^{\pi}_{x_0} \in \ell_2(\pi)$, the spectral theorem yields
\begin{equation*}
    \delta^{\pi}_{x_0} = \sum_{i \geq 1} \langle \delta^{\pi}_{x_0}, f_i \rangle_\pi f_i,
\end{equation*}
and
\begin{equation*}
    P^t \delta^{\pi}_{x_0} = \sum_{i \geq 1} \langle \delta^{\pi}_{x_0}, f_i \rangle_\pi (P^t f_i) = \sum_{i \geq 1} \lambda_i^t \langle \delta^{\pi}_{x_0}, f_i \rangle_\pi f_i.
\end{equation*}
Thus, for $x_0' \in \calX$,
\begin{equation}
\label{eq:spectral-expansion}
\begin{split}
    \langle P^t \delta^{\pi}_{x_0}, \delta^{\pi}_{x'_0} \rangle_\pi &= \sum_{i \geq 1} \lambda_i^t \langle \delta^{\pi}_{x_0}, f_i \rangle_\pi \langle f_i, \delta^{\pi}_{x'_0} \rangle_\pi = \sum_{i \geq 1} \lambda_i^t f_i(x_0) f_i(x_0'),
\end{split}
\end{equation}
which is convergent from \eqref{eq:convergent-series}.
On the other hand,
\begin{equation}
\label{eq:computation}
    \begin{split}
        \langle P^t \delta^{\pi}_{x_0}, \delta^{\pi}_{x'_0} \rangle_\pi &= \sum_{x \in \calX} \left( P^t \delta^{\pi}_{x_0} \right)(x) \delta^{\pi}_{x'_0}(x) \pi(x) \\
        &= \sum_{x \in \calX} \left( \sum_{x' \in \calX} P^t(x,x') \delta^{\pi}_{x_0}(x') \right) \delta^{\pi}_{x'_0}(x) \pi(x) \\
        &= \sum_{x \in \calX} \frac{P^t(x,x_0)}{\pi(x_0)} \delta^{\pi}_{x'_0}(x) \pi(x) = \frac{P^t(x_0',x_0)}{\pi(x_0)} \\
        &= \frac{P^t(x_0,x'_0)}{\pi(x'_0)},
    \end{split}
\end{equation}
where the last equality follows from reversibility of $P$. As a result of \eqref{eq:spectral-expansion} and \eqref{eq:computation}, for any $x_0, x_0' \in \calX$, the following representation holds
\begin{equation*}
    \frac{P^t(x_0,x'_0)}{\pi(x'_0)} = \sum_{i \geq 1} \lambda_i^t f_i(x_0) f_i(x_0') = 1 + \sum_{i \geq 2} \lambda_i^t f_i(x_0) f_i(x_0').
\end{equation*}
It follows that
\begin{equation*}
\begin{split}
    \abs{\frac{P^t(x_0,x'_0)}{\pi(x'_0)} - 1} &= \abs{\sum_{i \geq 2} \lambda_i^t f_i(x_0) f_i(x_0')} \\
    &\leq (1 - \gamma_\star)^t\sum_{i \geq 1} \abs{f_i(x_0) f_i(x_0')} \\
    &\leq \frac{(1 - \gamma_\star)^t}{\sqrt{\pi(x_0) \pi(x'_0)}},\\
\end{split}
\end{equation*}
whence ($\spadesuit$) holds.

\qed

\subsection{Proof of Theorem~\ref{theorem:estimation-beta-coefficients-skip-wise-without-tmix}}
The proof of Theorem~\ref{theorem:estimation-beta-coefficients-skip-wise-without-tmix} will rely on the following lemma.
\begin{lemma}
    \label{lemma:variance-markov-visit-counts}
Let $X_1, \dots, X_n$ be a Markov chain over a countable state space $\calX$ with stationary distribution $\pi$. For any $x \in \calX$, and any $p \in [1, \infty)$, it holds that
    \begin{equation*}
    \begin{split}
\Var[\pi]{N_{x}} &\leq 4 (n - 1) \pi(x)^{1-1/p} \left(\frac{\beta(0)^{1/p}}{2} + \sum_{t = 1}^{n - 2} \beta(t)^{1/p} \right) \leq 4 (n - 1) \pi(x)^{1-1/p} \calB_p,
    \end{split}
\end{equation*}
where $N_x = \sum_{t=1}^{n-1} \pred{X_t = x}$
counts the number of visits to state $x \in \calX$.
\end{lemma}

\begin{proof}
    The sequence $(\pred{X_t = x})_{t \in \bbN}$
being stationary at second order,
we expand the variance of the sum \citep[(1.6)]{rio1999theorie} as,
\begin{equation}
\label{eq:sum-variance-decomposition}
    \Var[\pi]{N_{x}} = (n - 1) \Var[\pi]{\pred{X_1 = x}} + 2 \sum_{t = 1}^{n - 2} (n - 1 - t) \Cov_\pi(\pred{X_1 = x}, \pred{X_{1 + t} = x}).
\end{equation}
From \citet[Theorem~1]{rio1993covariance}\footnote{This generalization of Ibragimov's covariance inequality \citep{ibragimov1962some} also appears in \citet{davydov1968convergence} but with a weaker constant.}---see also \citet[(1.12b)]{rio1999theorie} and \citet[Theorem~3]{doukhan2012mixing}---for a H\"{o}lder triplet $p, r, q \geq 1, \frac{1}{p} + \frac{1}{q} + \frac{1}{r} = 1$ and any $t \in \bbN \cup \set{0}$,
\begin{equation*}
\begin{split}
    \Cov_\pi(\pred{X_1 = x}, \pred{X_{1 + t} = x}) &\leq 2 (2\alpha(t))^{1/p} \nrm{\pred{X_1 = x}}_{\pi, r} \nrm{\pred{X_t = x}}_{\pi, q} \\
    &\leq 2 \beta(t)^{1/p} \pi(x)^{1 - 1/p},
\end{split}
\end{equation*}
where $\alpha$ is Rosenblatt's strong mixing coefficient\footnote{We note that $\alpha$ is defined as twice Rosenblatt's coefficient in \citet{rio1999theorie}.} \citep{rosenblatt1956central}, we wrote $\nrm{X}_{\pi, r} = \left(\bbE_\pi \abs{X}^r\right)^{1/r}$ for a real random variable $X$, and for the second inequality, we relied on the inequality $2 \alpha(t) \leq \beta(t)$ \citep[(1.11)]{bradley2005basic}.
Plugging in \eqref{eq:sum-variance-decomposition}, it follows that
\begin{equation*}
    \begin{split}
\Var[\pi]{N_{x}} &\leq 2(n - 1) \beta(0)^{1/p} \pi(x)^{1 - 1/p} + 4 \sum_{t = 1}^{n - 2} (n - 1 - t)  \beta(t)^{1/p} \pi(x)^{1 - 1/p} \\
&\leq 4 (n - 1) \pi(x)^{1-1/p} \left(\frac{\beta(0)^{1/p}}{2} + \sum_{t = 1}^{n - 2} \beta(t)^{1/p} \right).
    \end{split}
\end{equation*}
\end{proof}

Moving on to the proof of Theorem~\ref{theorem:estimation-beta-coefficients-skip-wise-without-tmix},
for $s \in \bbN$, by definition of $\widehat{\beta}(s)$ in \eqref{eq:estimator-beta-s},
\begin{equation*}
\begin{split}
    2 \bbE_\pi \abs{ \widehat{\beta}(s) - \beta(s) } =& \bbE_\pi \abs{ \frac{1}{\floor{(n-1)/s}} \sum_{x, x' \in \calX}  \abs{N_{x x'}^{(s)} -  \frac{N^{(s)}_x N^{(s)}_{x'}}{\floor{(n-1)/s}}} - \sum_{x \in \calX} \pi(x) \nrm{ e_x P^s - \pi }_1 } \\
    \leq&  \bbE_\pi  \sum_{x, x' \in \calX}  \abs{\frac{N_{x x'}^{(s)}}{\floor{(n-1)/s}} -    \pi(x)P^s(x,x') } \\ 
    &+ \bbE_\pi  \sum_{x, x' \in \calX} \abs{ \frac{N^{(s)}_x N^{(s)}_{x'}}{\floor{(n-1)/s}^2} - \pi(x)\pi(x')}.
\end{split}
\end{equation*}
With a symmetry argument, we further decompose the argument in the second expectation as
\begin{equation*}
\begin{split}
    \sum_{x, x' \in \calX} \abs{ \frac{N^{(s)}_x N^{(s)}_{x'}}{\floor{(n-1)/s}^2} - \pi(x)\pi(x')}
    &\leq 2\sum_{x \in \calX} \abs{ \frac{N^{(s)}_x}{\floor{(n-1)/s}} - \pi(x)}.
\end{split}
\end{equation*}
By stationarity,
\begin{equation*}
    \E[\pi]{\frac{N^{(s)}_x}{\floor{(n-1)/s}}} = \pi(x), \qquad \E[\pi]{\frac{N_{x x'}^{(s)}}{\floor{(n-1)/s}}} = \pi(x)P^s(x,x'),
\end{equation*}
and it follows from Jensen's inequality that
    \begin{equation}
    \label{eq:absolute-deviation-variance-bound}
\begin{split}
    2 \bbE_\pi \abs{ \widehat{\beta}(s) - \beta(s) } \leq \frac{1}{\floor{(n-1)/s}} \left( \sum_{x, x' \in \calX} \sqrt{\Var[\pi]{N_{xx'}^{(s)}}} + 2 \sum_{x \in \calX} \sqrt{\Var[\pi]{N_{x}^{(s)}}} \right).
\end{split}
\end{equation}
We proceed to bound the two variance terms.
From Lemma~\ref{lemma:variance-markov-visit-counts}, for a fixed $x \in \calX$ and any $s < n$, 
noting that the mixing coefficient $\beta^{(s)}(t)$ pertaining to the $s$-skipped chain satisfies $\beta^{(s)}(t) = \beta(st)$, we obtain
\begin{equation*}
    \begin{split}
    \Var[\pi]{N_{x}^{(s)}} 
    &\leq 4 \floor*{\frac{n-1}{s}} \pi(x)^{1 - 1/p} \calB^{(s)}_p. \\
    \end{split}
\end{equation*}
The process consisting of a sliding window over pairs,
\begin{equation*}
    \{E_t\}_{t \in \bbN} \eqdef \set{ \left( X_{t}, X_{t +1} \right) }_{t \in \bbN}    
\end{equation*}
forms a Markov chain \citep[Lemma~6.1]{wolfer2021} over the state space 
$$\calE = \set{(x,x') \in \calX^2 \colon P(x,x') > 0 },$$ called the Hudson expansion \citep{kemeny1983finite, wolfer2024geometric} of the original chain, with transition operator
\begin{equation*}
    \overline{P}(e = (x_1, x_2), e' = (x_1', x_2') ) = \pred{ x_2 = x_1' } P(x_2, x_2'), \forall (e,e') \in \calE^2,
\end{equation*}
and with stationary distribution $\overline{\pi}((x,x')) = \pi(x)P(x,x')$.
Writing $\overline{\beta}(s)$ for the $\beta$-mixing coefficient of $\overline{P}$, we compute
\begin{equation*}
    \begin{split}
    \overline{\beta}(0) &= 1 - \nrm{\overline{\pi}}_2^2, \\
    \overline{\beta}(1) &= \beta(0) = 1 - \nrm{\pi}_2^2.
    \end{split}
\end{equation*}
Moreover, it follows from the Markov property that for $t \in \bbN$,
$$\overline{P}^{t+1}((x_1, x_2), (x_1', x_2')) = P(x_1', x_2') P^t(x_2, x_1'),$$
and thus,
\begin{equation*}
\begin{split}
    \overline{\beta}(t+1) &= \frac{1}{2}\sum_{e = (x_1, x_2) \in \calE} \overline{\pi}(e) \sum_{e' = (x_1', x_2') \in \calE} \abs{\overline{P}^{t+1}(e, e') - \overline{\pi}(e')} \\
    &= \frac{1}{2}\sum_{(x_1, x_2) \in \calE} \pi(x_1) P(x_1, x_2) \sum_{(x'_1, x'_2) \in \calE} \abs{ P^t(x_2, x_1') - \pi(x_1')}P(x_1', x_2') \\
    &= \frac{1}{2}\sum_{x_2 \in \calX} \pi(x_2) \sum_{x'_1 \in \calX} \abs{ P^t(x_2, x_1') - \pi(x_1')} = \beta(t).
\end{split}
\end{equation*}
More generally, for $s \in \bbN, t \in \bbN \cup \set{0}$, $1 \leq s \leq n - 1$, and $0 \leq t < \floor{(n-1) / s}$,
\begin{equation*}
    \overline{\beta}^{(s)}(t + 1) = \beta^{(s)}(t) = \beta(st).
\end{equation*}
An application of Lemma~\ref{lemma:variance-markov-visit-counts} on the Hudson expanded chain yields that
\begin{equation*}
    \Var[\pi]{N_{xx'}^{(s)}} \leq 4 \floor*{\frac{n-1}{s}} Q^{(s)}(x, x')^{1 - 1/p} \left( \frac{\overline{\beta}^{(s)}(0)^{1/p}}{2} + \sum_{t = 1}^{\floor{(n-1)/s} - 1} \overline{\beta}^{(s)}(t)^{1/p} \right),
\end{equation*}
where we recall that $Q^{(s)}(x,x') = \pi(x)P^s(x,x')$. We further bound
\begin{equation*}
\begin{split}
   \frac{\overline{\beta}^{(s)}(0)^{1/p}}{2} + \sum_{t = 1}^{\floor{(n-1)/s} - 1} \overline{\beta}^{(s)}(t)^{1/p} &=  \frac{1}{2}\left( 1 - \nrm{Q^{(s)}}_2^2 \right)^{1/p} + \sum_{t = 1}^{\floor{(n-1)/s} - 1} \beta(s(t - 1))^{1/p} \\ 
    &\leq 1/2 + \sum_{t = 0}^{\floor{(n-1)/s} - 2} \beta(st)^{1/p} \leq 1/2 + \calB_{p}^{(s)},
    \end{split}
\end{equation*}
and plugging in \eqref{eq:absolute-deviation-variance-bound},
\begin{equation*}
\begin{split}
    \bbE_{\pi} \abs{ \widehat{\beta}(s) - \beta(s) } 
    &\leq \sqrt{\frac{1}{\floor{(n-1)/s}}} \left( \nrm{Q^{(s)}}_{(1 - 1/p)/2}^{(1 - 1/p)/2} \sqrt{1/2 + \calB_p^{(s)}} + 2  \nrm{\pi}_{(1 - 1/p)/2}^{(1 - 1/p)/2} \sqrt{\calB_p^{(s)} }\right).
\end{split}
\end{equation*}
Finally, the claim can be obtained from observing that $\nrm{Q^{(s)}}_{(1 - 1/p)/2} \geq \nrm{\pi}_{(1 - 1/p)/2}$.
\qed

\subsection{Proof of Lemma~\ref{lemma:bound-mixing-factor}}
Let us first analyze the exponential case, by assuming that $\beta(s) \leq \beta_0 \exp(-\beta_1 s)$. 
\begin{equation*}
\begin{split}
    \calB_{p}^{(s)} \leq \beta_0^{1/p} \sum_{t = 0}^{\floor{(n-1)/s} - 1}  \exp\left( - \beta_1 s t / p \right) &\leq \frac{\beta_0^{1/p}}{1 - \exp(- \beta_1 s / p)}.
\end{split}
\end{equation*}
We proceed to address the sub-exponential case.
Let us assume that $\beta(s) \leq \beta_0 \exp(-\beta_1 s^b)$ for $\beta_0 > 0, \beta_1 > 0$ and $b \in (0,1)$. Then,
\begin{equation*}
\begin{split}
    \calB_p^{(s)} &\leq \beta_0^{1/p} + \beta_0^{1/p} \sum_{t = 1}^{\floor{(n-1)/s} - 1}  \exp\left( - \beta_1 (s t)^b / p \right) \\
    &\leq \beta_0^{1/p} + \beta_0^{1/p} \exp(-\beta_1 s^b/p) + \beta_0^{1/p}  \sum_{t = 2}^{\infty} \exp\left( - \beta_1 s^b t^b / p \right) \\
    &\leq \beta_0^{1/p} + \beta_0^{1/p} \exp(-\beta_1 s^b/p) + \beta_0^{1/p} \int_{1}^{\infty} \exp(- \beta_1 s^b t^b/p) dt \\
    &= \beta_0^{1/p} + \beta_0^{1/p} \exp(-\beta_1 s^b/p) + \beta_0^{1/p} \left[ \frac{-1}{b (\beta_1 s^b / p)^{1/b}} \Gamma\left(\frac{1}{b}, \beta_1 s^b t^b / p\right) \right]_{1}^{\infty}.
\end{split}
\end{equation*}

Finally, in the polynomial case, where for $s \geq 1$, $\beta(s) \leq \beta_1 /s^b$,
\begin{equation*}
\begin{split}
    \calB_p^{(s)}
    &\leq
    \beta_0^{1/p}
    +
    \frac{\beta_1^{1/p}}{s^{b/p}}
    \sum_{t=1}^{\floor{(n-1)/s}-1}
    \frac{1}{t^{b/p}} \leq
    \beta_0^{1/p}
    +
    \zeta(b/p)
    \frac{\beta_1^{1/p}}{s^{b/p}}.
\end{split}
\end{equation*}
\qed

\subsection{Proof of Theorem~\ref{theorem:estimation-beta-coefficients-general-ergodicity-skip-wise}}
We first focus on the estimation of $\beta(1)$. The general result for $\beta(s)$ will be obtained by considering an $s$-skipped chain.
First, similar to the proof of Theorem~\ref{theorem:estimation-beta-coefficients-skip-wise-without-tmix}, we decompose using the triangle inequality,
\begin{equation*}
\begin{split}
        2 \abs{\widehat{\beta}(1) - \beta(1)} 
        &\leq \sum_{x, x' \in \calX} \abs{    \frac{N_{x x'}}{n-1}   -   Q(x, x')  } + 2 \sum_{x \in \calX} \abs{\frac{N_x}{n-1} - \pi(x)}.
\end{split}
\end{equation*}
We treat the two sums separately,
\begin{equation}
\label{eq:two-probability-terms}
    \PR[\pi]{\abs{\widehat{\beta}(1) - \beta(1)} > \eps} \leq \PR[\pi]{\sum_{x, x' \in \calX} \abs{    \frac{N_{x x'}}{n-1}   -   Q(x, x')  } >  \eps} + \PR[\pi]{\sum_{x \in \calX} \abs{\frac{N_x}{n-1} - \pi(x)} >  \eps/2}.
\end{equation}
Let us focus on the second probability statement.
For simplicity, as in the proof of
Lemma~\ref{lemma:large-deviation-bound-average-mixing-time}, we assume
that $n-1=2Br$ and decompose the process into $2B$ Bernstein blocks of
size $r$ (refer to Figure~\ref{fig:blocking-method}); for general $n$,
using blocks of lengths $r$ or $r+1$ and length-weighted averages only
changes the universal constants.
\begin{equation*}
    \begin{split}
        \sum_{x \in \calX} \abs{    \frac{N_{x}}{n-1}   -   \pi(x) } 
         &= \sum_{x \in \calX} \abs{\frac{1}{2} \sum_{\sigma \in \{0,1\}}  \frac{1}{B} \sum_{b = 1}^{B} \frac{1}{r} \sum_{t=1}^{r}    
         \left(  \pred{X_{t}^{[2b - \sigma]} = x}   -   \pi(x) \right)} \\
          &\leq \frac{1}{2}\sum_{\sigma \in \{0,1\}} \frac{1}{B} \sum_{b = 1}^{B} \frac{1}{r} \sum_{x \in \calX}  \abs{     
          \sum_{t=1}^{r} \left(\pred{X_{t}^{[2b - \sigma]} = x }   -   \pi(x)\right) }. \\
    \end{split}
\end{equation*}
Therefore, applying \citet[Corollary~2.7]{yu1994rates},
\begin{equation}
\label{eq:after-blocking}
    \PR[\pi]{\sum_{x \in \calX} \abs{\frac{N_x}{n-1} - \pi(x)} >  \eps/2} \leq \sum_{\sigma \in \{0,1\}} \PR[\pi]{ \frac{1}{B} \sum_{b = 1}^{B} Z_b^{(\sigma)} >  \eps/2} + 2(B-1) \beta(r),
\end{equation}
where
\begin{equation*}
    Z_b^{(\sigma)} \eqdef \frac{1}{r} \sum_{x \in \calX}  \abs{     
          \sum_{t=1}^{r} \left(\pred{X_{t}^{[2b - \sigma]} = x }   -   \pi(x)\right) }
\end{equation*}
and where, on the right-hand side of
\eqref{eq:after-blocking}, these variables are evaluated under the
independent-block product distribution.
First, observing that $Z_b^{(\sigma)}$ is a function of a Markov chain $X_1, \dots, X_r$ stationarily sampled from $P$, we upper bound the expected value as in Theorem~\ref{theorem:estimation-beta-coefficients-skip-wise-without-tmix},
\begin{equation*}
\begin{split}    
    \bbE_\pi Z_b^{(\sigma)} &\stackrel{(i)}{\leq} \frac{1}{r} \sum_{x \in \calX}   \sqrt{ \Var[\pi]{\sum_{t=1}^{r} \pred{X_{t}^{[2b - \sigma]} = x}} } \\
          &\stackrel{(ii)}{\leq} \frac{2}{\sqrt{r}} \nrm{\pi}_{(1 - 1/p)/2}^{(1 - 1/p)/2}  \sqrt{ \calB_{p}},
\end{split}
\end{equation*}
where $(i)$ is Jensen's inequality and $(ii)$ is an application of Lemma~\ref{lemma:variance-markov-visit-counts}.
Thus for
\begin{equation*}
    r \geq \frac{64 \nrm{\pi}_{(1-1/p)/2}^{1-1/p} \calB_p}{\eps^2},
\end{equation*}
it holds that
\begin{equation*}
    \bbE_\pi Z_b^{(\sigma)} \leq \eps/4,
\end{equation*}
and it remains to bound
\begin{equation*}
    \PR[\pi]{ \frac{1}{B} \sum_{b = 1}^{B} Z_b^{(\sigma)} - \bbE_{\pi} Z_b^{(\sigma)}  >  \eps/4}.
\end{equation*}
The random variables $Z_b^{(\sigma)}$ are bounded. Indeed,
\begin{equation*}
\begin{split}
    Z^{(\sigma)}_b &= \frac{1}{r} \sum_{x \in \calX}  \abs{     
          \sum_{t=1}^{r} \left(\pred{X_{t}^{[2b - \sigma]} = x}   -   \pi(x)\right) } \\
          &\leq \frac{1}{r} \left( \sum_{x  \in \calX} \sum_{t = 1}^{r} \pred{X_{t}^{[2b - \sigma]} = x} +  \sum_{x \in \calX} \sum_{t = 1}^{r} \pi(x) \right) \leq 2. \\
\end{split}
\end{equation*}
Therefore, from an application of Hoeffding's inequality,
\begin{equation*}
    \PR[\pi]{ \frac{1}{B} \sum_{b = 1}^{B} Z_b^{(\sigma)} - \bbE_{\pi} Z_b^{(\sigma)}  >  \eps/4} \leq \exp\left( -\frac{2(B\eps/4)^2}{B 2^2} \right) = \exp\left( -\frac{B\eps^2}{32} \right).
\end{equation*}
Plugging in \eqref{eq:after-blocking},
\begin{equation*}
    \PR[\pi]{\sum_{x \in \calX} \abs{\frac{N_x}{n-1} - \pi(x)} >  \eps/2} \leq 2\exp\left( -\frac{B\eps^2}{32} \right) + 2 (B-1) \beta(r).
\end{equation*}
For $B = \left\lceil
        \frac{32\log(8/\delta)}{\eps^2}
    \right\rceil$ and $r \geq \atmix(\xi(\eps, \delta))$ with 
$$\xi(\eps,\delta) = \frac{ \eps^2 \delta}{256 \log(8/\delta)},$$
the above probability is smaller than $\delta/2$.
In other words, regardless of uniform ergodicity, for
\begin{equation*}
    n \geq 1 + \frac{128 \log(8/\delta)}{\eps^2} \max \set{ \frac{64}{\eps^2} \nrm{\pi}_{(1 - 1/p)/2}^{1 - 1/p} \calB_p, \atmix(\xi(\eps, \delta)) }, \qquad \xi(\eps,\delta) = \frac{ \eps^2 \delta}{256 \log(8/\delta)},
\end{equation*}
it holds that
\begin{equation*}
    \PR[\pi]{\sum_{x \in \calX} \abs{\frac{N_x}{n-1} - \pi(x)} >  \eps/2} \leq \delta/2.
\end{equation*}
A similar approach for the sliding-window chain $(X_t, X_{t+1})_{t \in \bbN}$ leads to a bound on the first term of
\eqref{eq:two-probability-terms}.
Finally, applying the result to the $s$-skipped chain and using
\begin{equation*}
s\left(1+{\atmix}^{(s)}(\xi)\right)
\leq
s\left(1+\ceil{\atmix(\xi)/s}\right)
\leq
\atmix(\xi)+2s,
\end{equation*}
yields the claim for $\beta(s)$ after adjusting the universal constant.
\qed

\subsection{Proof of Theorem~\ref{theorem:estimation-beta-coefficients-uniform-ergodicity-skip-wise-amd}}
Our starting point for this proof is 
\eqref{eq:absolute-deviation-variance-bound}.
However, instead of \citeauthor{rio1993covariance}'s inequality, we rely on techniques in \citet{paulin2015concentration} to obtain an upper bound on the two variance terms. The mixing time of the $s$-skipped pairs of observations is upper bounded by $\tmix^{(s)} + 1 \leq \ceil{\tmix/s} + 1$ \citep[Lemma~6.1]{wolfer2021}.
Combining \citet[Theorem~3.2, Proposition ~3.4]{paulin2015concentration}, it follows that
\begin{equation*}
\begin{split}
    \Var[\pi]{N_{x}^{(s)}} &\leq  4 \floor{(n-1)/s} \ceil{\tmix /s}  \pi(x), \\
    \Var[\pi]{N_{xx'}^{(s)}} &\leq  4 \floor{(n-1)/s} (\ceil{\tmix /s} + 1) Q^{(s)}(x,x'), \\
\end{split}
\end{equation*}
which yields
\begin{equation*}
\begin{split}
    \bbE \abs{ \widehat{\beta}(s) - \beta(s) } 
    \leq& \sqrt{\frac{\ceil{\tmix /s} + 1}{\floor{(n-1)/s}}} \left( \nrm{Q^{(s)}}_{1/2}^{1/2} + 2 \nrm{\pi}_{1/2}^{1/2} \right) \\
    \leq& 3 \sqrt{\frac{\ceil{\tmix /s} + 1}{\floor{(n-1)/s}} \calJ_{\infty}^{(s)}}. \\
\end{split}
\end{equation*}
The claim with $\Cue \leq 3 \sqrt{2}$ follows from $n > s + 1$ and $x > 1 \implies \floor{x} \geq x/2$ and
$s(\ceil{\tmix /s} + 1) \leq \tmix + 2s$.

\qed

\subsection{Proof of Theorem~\ref{theorem:estimation-beta-coefficients-uniform-ergodicity-skip-wise-pac}}
From Theorem~\ref{theorem:estimation-beta-coefficients-uniform-ergodicity-skip-wise-amd}, when 
$$n \geq 1 + \frac{4\Cue^2}{\eps^2} (\tmix + 2s) \calJ_\infty^{(s)},$$
it holds that, $\bbE_\pi \abs{ \widehat{\beta}(s) - \beta(s) } \leq \eps /2$ and in this case,
the probability of the above deviation can be upper bounded by fluctuations around the mean as follows,
\begin{equation*}
    \PR[\pi]{\abs{\widehat{\beta}(s) - \beta(s)} > \eps} \leq \PR[\pi]{\abs{\abs{\widehat{\beta}(s) - \beta(s)} - \bbE_\pi \abs{\widehat{\beta}(s) - \beta(s)} }> \eps / 2}.
\end{equation*}

Let $$\bar{x} = (x_{1}, x_{2}, \dots, x_{n}) \in \calX^n,$$ and for $t \in [n]$ and $x'_{t} \in \calX$, let us
denote 
$$\bar{x}^{\set{t}} \eqdef (x_{1}, x_2, \dots, x_{t-1}, x'_{t}, x_{t+1} \dots, x_{n})$$ the trajectory where the observation $x_{t}$ at time $t$ was replaced with observation $x'_{t}$.
We show that the absolute deviation $\abs{\widehat{\beta}(s) - \beta(s)} \colon \calX^n \to \bbR$ satisfies the bounded differences condition. We decompose, 
\begin{equation*}
\begin{split}
    &\abs{\abs{\widehat{\beta}(s)(\bar{x}) - \beta(s)} - \abs{\widehat{\beta}(s)(\bar{x}^{\set{t}}) - \beta(s)}} \\ &\leq \abs{\widehat{\beta}(s)(\bar{x}) - \widehat{\beta}(s)(\bar{x}^{\set{t}})} \\
    &\leq \frac{1}{\floor{(n-1)/s}} \sum_{x,x' \in \calX} \abs{ N_{xx'}^{(s)}(\bar{x}) - N_{xx'}^{(s)}(\bar{x}^{\set{t}}) } \\ &+ \frac{1}{\floor{(n-1)/s}^2} \sum_{x,x' \in \calX} \abs{ N_x^{(s)}(\bar{x}) N_{x'}^{(s)}(\bar{x}) - N_x^{(s)}(\bar{x}^{\set{t}}) N_{x'}^{(s)}(\bar{x}^{\set{t}})}.\\
\end{split}
\end{equation*}
We further upper bound the second term as follows,
\begin{equation*}
\begin{split}
    & \frac{1}{\floor{(n-1)/s}^2} \sum_{x,x' \in \calX} \abs{ N_x^{(s)}(\bar{x}) N_{x'}^{(s)}(\bar{x}) - N_x^{(s)}(\bar{x}^{\set{t}}) N_{x'}^{(s)}(\bar{x}^{\set{t}})} \\
    &\leq \frac{1}{\floor{(n-1)/s}^2} \sum_{x,x' \in \calX} \abs{ N_x^{(s)}(\bar{x}) N_{x'}^{(s)}(\bar{x}) - N_x^{(s)}(\bar{x}) N_{x'}^{(s)}(\bar{x}^{\set{t}})} \\
    &+ \frac{1}{\floor{(n-1)/s}^2} \sum_{x,x' \in \calX} \abs{ N_x^{(s)}(\bar{x}) N_{x'}^{(s)}(\bar{x}^{\set{t}}) - N_x^{(s)}(\bar{x}^{\set{t}}) N_{x'}^{(s)}(\bar{x}^{\set{t}})}  \\
    &\leq \frac{1}{\floor{(n-1)/s}} \sum_{x' \in \calX} \abs{  N_{x'}^{(s)}(\bar{x}) -  N_{x'}^{(s)}(\bar{x}^{\set{t}})} + \frac{1}{\floor{(n-1)/s}} \sum_{x \in \calX} \abs{ N_x^{(s)}(\bar{x})  - N_x^{(s)}(\bar{x}^{\set{t}}) }.\\
\end{split}
\end{equation*}
We obtain
\begin{equation*}
    \abs{\abs{\widehat{\beta}(s)(\bar{x}) - \beta(s)} - \abs{\widehat{\beta}(s)(\bar{x}^{\set{t}}) - \beta(s)}} \leq \frac{4}{\floor{(n-1)/s}} + \frac{2}{\floor{(n-1)/s}} + \frac{2}{\floor{(n-1)/s}}.
\end{equation*}
From a version of McDiarmid's inequality for uniformly ergodic Markov chains \citep[Corollary~2.10]{paulin2015concentration}, and the fact that the mixing time of the $s$-skipped chain $\tmix^{(s)}$ verifies $\tmix^{(s)} \leq \ceil{\tmix/s}$,
\begin{equation*}
\begin{split}
    \PR[\pi]{ \abs{\widehat{\beta}(s) - \beta(s)} > \eps } &\leq 2 \exp \left( - \frac{2 (\eps/2)^2}{9 \tmix^{(s)} 2 \floor{(n-1)/s} \left( \frac{8}{\floor{(n-1)/s}} \right)^2} \right) \\ &\leq 2 \exp \left( - \frac{(n-1) \eps^2}{9 \cdot 2^{9}(\tmix + s)} \right),
\end{split}
\end{equation*}
which completes the proof.
\qed

\subsection{Proof of Lemma~\ref{lemma:convert-beta-pac-to-atmix-bound}}
We first bound the probability that the estimator overshoots the target interval, that is $\widehat{\tmix}^{\sharp}(\xi) > \tmix^{\sharp}(\xi(1 - \eps))$. 
On this event, for any $s \leq \tmix^{\sharp}(\xi(1 - \eps))$, and in particular for $s = \tmix^{\sharp}(\xi(1 - \eps))$, it must be that $\widehat{\beta}(s) > \xi$. 
From our assumption
\begin{equation*}
    n \geq n_0(P, \atmix(\xi(1- \eps)), \xi\eps, \delta/2),
\end{equation*}
it holds that
\begin{equation*}
\begin{split}
\PR[\pi]{ \widehat{\beta}(\tmix^{\sharp}(\xi(1 - \eps))) > \xi } 
&\leq \PR[\pi]{ \widehat{\beta}(\tmix^{\sharp}(\xi(1 - \eps))) - \beta(\tmix^{\sharp}(\xi(1 - \eps))) > \xi - \xi(1 - \eps) } \\
    &\leq \PR[\pi]{ \abs{\widehat{\beta}(\tmix^{\sharp}(\xi(1 - \eps))) - \beta(\tmix^{\sharp}(\xi(1 - \eps)))} > \xi \eps } \leq \frac{\delta}{2}.
\end{split}
\end{equation*}
As a second step, we bound the probability that $\widehat{\tmix}^\sharp(\xi)$ underestimates the true parameter, that is $\widehat{\tmix}^{\sharp}(\xi) < \tmix^{\sharp}(\xi(1 + \eps))$.
From our assumption and an application of the union bound, when for any $s \leq \atmix(\xi(1 + \eps))$ the trajectory length satisfies
\begin{equation*}
    n \geq n_0\left(P, s, \xi \eps, \frac{\delta}{2 \atmix(\xi)}\right),
\end{equation*}
it holds that
\begin{equation*}
\begin{split}
    \PR[\pi]{\exists s < \atmix(\xi(1 + \eps)), \widehat{\beta}(s) \leq \xi} &\leq \sum_{s = 1}^{\atmix(\xi(1 + \eps)) - 1} \PR[\pi]{ \widehat{\beta}(s) \leq \xi}\\
    &\leq \sum_{s = 1}^{\atmix(\xi(1 + \eps)) - 1} \PR[\pi]{ \beta(s) - \widehat{\beta}(s) \geq -\xi + \xi(1 + \eps)}\\
    &\leq \sum_{s = 1}^{\atmix(\xi(1 + \eps)) - 1} \PR[\pi]{ \abs{\widehat{\beta}(s) - \beta(s)} \geq \xi \eps} \leq \frac{\delta}{2}.
\end{split}
\end{equation*}
\qed

\section*{Acknowledgments}
GW would like to thank Luca Zanetti and John Sylvester for posing the original question of estimating the average-mixing time,
and Aryeh Kontorovich for the interesting discussions. The authors also thank the anonymous referees for their careful reading and insightful remarks. The authors used ChatGPT (OpenAI) to assist with proofreading and take full responsibility for the final content.

\section*{Funding}
The work of GW was supported in part by the SPDR Program of RIKEN and by Waseda University Grants for Special Research Projects (Project number: 2024C-667).
The visit of PA to RIKEN in October 2023, during which part of this work was conducted, was sponsored by the SPDR program of RIKEN.

\bibliography{bibliography}
\bibliographystyle{abbrvnat}

\end{document}